\documentclass[12pt]{amsart}
\usepackage{}
\usepackage{amsmath}
\usepackage{amsfonts}
\usepackage{amssymb}
\usepackage[all,cmtip]{xy}           
\usepackage{xspace}

\usepackage{bbding}
\usepackage{txfonts}
\usepackage[shortlabels]{enumitem}
\usepackage{ifpdf}
\ifpdf
\usepackage[colorlinks,final,backref=page,hyperindex]{hyperref}
\else
\usepackage[colorlinks,final,backref=page,hyperindex,hypertex]{hyperref}
\fi
\usepackage{tikz}
\usepackage[active]{srcltx}

\makeatletter

\newtheorem{theorem}{Theorem}[section]
\newtheorem{prop}[theorem]{Proposition}
\newtheorem{lemma}[theorem]{Lemma}
\newtheorem{coro}[theorem]{Corollary}
\newtheorem{prop-def}{Proposition-Definition}[section]

\theoremstyle{definition}
\newtheorem{defn}[theorem]{Definition}

\newtheorem{remark}[theorem]{Remark}

\newcommand{\nc}{\newcommand}

\nc{\delete}[1]{{}}
\nc{\mmargin}[1]{}

\nc{\mlabel}[1]{\label{#1}}  
\nc{\mcite}[1]{\cite{#1}}  
\nc{\mref}[1]{\ref{#1}}  
\nc{\meqref}[1]{\eqref{#1}} 
\nc{\mbibitem}[1]{\bibitem{#1}} 

\delete{
\nc{\mlabel}[1]{\label{#1}  
		{\hfill \hspace{1cm}{\bf{{\ }\hfill(#1)}}}}
\nc{\mcite}[1]{\cite{#1}{{\bf{{\ }(#1)}}}}  
\nc{\mref}[1]{\ref{#1}{{\bf{{\ }(#1)}}}}  
\nc{\meqref}[1]{\eqref{#1}{{\bf{{\ }(#1)}}}} 
\nc{\mbibitem}[1]{\bibitem[\bf #1]{#1}} 
}

\nc{\lush}{dense\xspace}

 \font\cyrs=wncyr7

\nc{\tforall}{\text{ for all }}

\newcommand{\bk}{{\mathbf{k}}}

\nc{\vep}{\varepsilon}
\nc{\bin}[2]{ (_{\stackrel{\scs{#1}}{\scs{#2}}})}  
\nc{\binc}[2]{(\!\! \begin{array}{c} \scs{#1}\\
		\scs{#2} \end{array}\!\!)}  
\nc{\bincc}[2]{  ( {\scs{#1} \atop
		\vspace{-1cm}\scs{#2}} )}  
\nc{\oline}[1]{\overline{#1}}
\nc{\mapm}[1]{\lfloor\!|{#1}|\!\rfloor}
\nc{\bs}{\bar{S}}
\nc{\la}{\longrightarrow}
\nc{\ot}{\otimes}
\nc{\rar}{\rightarrow}
\nc{\dar}{\downarrow}
\nc{\dap}[1]{\downarrow \rlap{$\scriptstyle{#1}$}}
\nc{\defeq}{\stackrel{\rm def}{=}}
\nc{\dis}[1]{\displaystyle{#1}}
\nc{\dotcup}{\ \displaystyle{\bigcup^\bullet}\ }
\nc{\hcm}{\ \hat{,}\ }
\nc{\hts}{\hat{\otimes}}
\nc{\hcirc}{\hat{\circ}}
\nc{\lleft}{[}
\nc{\lright}{]}
\nc{\curlyl}{\left \{ \begin{array}{c} {} \\ {} \end{array}
	\right .  \!\!\!\!\!\!\!}
\nc{\curlyr}{ \!\!\!\!\!\!\!
	\left . \begin{array}{c} {} \\ {} \end{array}
	\right \} }
\nc{\longmid}{\left | \begin{array}{c} {} \\ {} \end{array}
	\right . \!\!\!\!\!\!\!}
\nc{\ora}[1]{\stackrel{#1}{\rar}}
\nc{\ola}[1]{\stackrel{#1}{\la}}
\nc{\scs}[1]{\scriptstyle{#1}} \nc{\mrm}[1]{{\rm #1}}
\nc{\dirlim}{\displaystyle{\lim_{\longrightarrow}}\,}
\nc{\invlim}{\displaystyle{\lim_{\longleftarrow}}\,}
\nc{\dislim}[1]{\displaystyle{\lim_{#1}}} \nc{\colim}{\mrm{colim}}
\nc{\mvp}{\vspace{0.3cm}} \nc{\tk}{^{(k)}} \nc{\tp}{^\prime}
\nc{\ttp}{^{\prime\prime}} \nc{\svp}{\vspace{2cm}}
\nc{\vp}{\vspace{8cm}}
\nc{\modg}[1]{\!<\!\!{#1}\!\!>}
\nc{\intg}[1]{F_C(#1)}
\nc{\lmodg}{\!<\!\!}
\nc{\rmodg}{\!\!>\!}
\nc{\cpi}{\widehat{\Pi}}
\nc{\ssha}{{\mbox{\cyrs X}}} 
\nc{\tsha}{{\mbox{\cyrt X}}}
\nc{\shpr}{\diamond}    
\nc{\labs}{\mid\!}
\nc{\rabs}{\!\mid}

\nc{\ann}{\mrm{ann}}
\nc{\Aut}{\mrm{Aut}}
\nc{\br}{\mrm{bre}}
\nc{\can}{\mrm{can}}
\nc{\Cont}{\mrm{Cont}}
\nc{\rchar}{\mrm{char}}
\nc{\cok}{\mrm{coker}}
\nc{\de}{\mrm{dep}}
\nc{\dtf}{{R-{\rm tf}}}
\nc{\dtor}{{R-{\rm tor}}}

\nc{\Div}{{\mrm Div}}
\nc{\End}{\mrm{End}}
\nc{\Ext}{\mrm{Ext}}
\nc{\Fil}{\mrm{Fil}}
\nc{\Fr}{\mrm{Fr}}
\nc{\Frob}{\mrm{Frob}}
\nc{\Gal}{\mrm{Gal}}
\nc{\GL}{\mrm{GL}}
\nc{\Hom}{\mrm{Hom}}
\nc{\hsr}{\mrm{H}}
\nc{\hpol}{\mrm{HP}}
\nc{\id}{\mrm{id}}
\nc{\im}{\mrm{im}}
\nc{\Id}{\mrm{Id}}
\nc{\ID}{\mrm{ID}}
\nc{\Irr}{\mrm{Irr}}
\nc{\incl}{\mrm{incl}}
\nc{\length}{\mrm{length}}
\nc{\Lie}{\mrm{Lie}}
\nc{\mchar}{\rm char}
\nc{\mpart}{\mrm{part}}
\nc{\ql}{{\QQ_\ell}}
\nc{\qp}{{\QQ_p}}
\nc{\rank}{\mrm{rank}}
\nc{\rcot}{\mrm{cot}}
\nc{\rdef}{\mrm{def}}
\nc{\rdiv}{{\rm div}}
\nc{\rtf}{{\rm tf}}
\nc{\rtor}{{\rm tor}}
\nc{\res}{\mrm{res}}
\nc{\SL}{\mrm{SL}}
\nc{\Spec}{\mrm{Spec}}
\nc{\tor}{\mrm{tor}}
\nc{\Tr}{\mrm{Tr}}
\nc{\tr}{\mrm{tr}}

\nc{\bfk}{{\bf k}}
\nc{\bfone}{{\bf 1}}
\nc{\bfzero}{{\bf 0}}
\nc{\detail}{\marginpar{\bf More detail}
	\noindent{\bf Need more detail!}
	\svp}
\nc{\Diff}{\mathbf{Diff}}
\nc{\gap}{\marginpar{\bf Incomplete}\noindent{\bf Incomplete!!}
	\svp}
\nc{\FMod}{\mathbf{FMod}}
\nc{\Int}{\mathbf{Int}}
\nc{\Mon}{\mathbf{Mon}}
\nc{\remarks}{\noindent{\bf Remarks: }}
\nc{\Rep}{\mathbf{Rep}}
\nc{\Rings}{\mathbf{Rings}}
\nc{\Sets}{\mathbf{Sets}}
\nc{\sbu}{{\scriptstyle \bullet}}

\nc{\BA}{{\mathbb A}}   \nc{\CC}{{\mathbb C}}
\nc{\DD}{{\mathbb D}}   \nc{\EE}{{\mathbb E}}
\nc{\FF}{{\mathbb F}}   \nc{\GG}{{\mathbb G}}
\nc{\HH}{{\mathbb H}}   \nc{\LL}{{\mathbb L}}
\nc{\NN}{{\mathbb N}}   \nc{\PP}{{\mathbb P}}
\nc{\QQ}{{\mathbb Q}}   \nc{\RR}{{\mathbb R}}
\nc{\TT}{{\mathbb T}}   \nc{\VV}{{\mathbb V}}
\nc{\ZZ}{{\mathbb Z}}   \nc{\TP}{\widetilde{P}}

\nc{\cala}{{\mathcal A}}    \nc{\calb}{{\mathcal B}}
\nc{\calc}{{\mathcal C}}    \nc{\cald}{\mathcal{D}}
\nc{\cale}{{\mathcal E}}    \nc{\calf}{{\mathcal F}}
\nc{\calg}{{\mathcal G}}    \nc{\calh}{{\mathcal H}}
\nc{\cali}{{\mathcal I}}    \nc{\call}{{\mathcal L}}
\nc{\calm}{{\mathcal M}}    \nc{\caln}{{\mathcal N}}
\nc{\calo}{{\mathcal O}}    \nc{\calp}{{\mathcal P}}
\nc{\calr}{{\mathcal R}}    \nc{\cals}{{\mathcal S}}
\nc{\calt}{{\mathcal T}}    \nc{\calv}{{\mathcal V}}
\nc{\calw}{{\mathcal W}}    \nc{\calx}{{\mathcal X}}
\nc{\caly}{{\mathcal Y}}

\nc{\fraka}{{\mathfrak a}}
\nc{\frakb}{\mathfrak{b}}
\nc{\frakg}{{\frak g}}
\nc{\frakB}{{\frak B}}
\nc{\frakm}{{\frak m}}
\nc{\frakM}{{\frak M}}
\nc{\frakp}{{\frak p}}
\nc{\frakX}{{\frak X}}
\nc{\frakS}{{\frak S}}
\nc{\frakA}{{\frak A}}
\nc{\frakC}{{\frak C}}
\nc{\frakx}{{\frakx}}

\nc{\ynr}[1]{\textcolor{blue}{\underline{Yunnan:}#1 }}

\nc{\lir}[1]{\textcolor{red}{\underline{Li:}#1 }}

\begin{document}

\title[Braided dendriform algebras and braided Hopf algebras]{Braided dendriform and tridendriform algebras and braided Hopf algebras of planar trees}

\author[Li Guo]{Li Guo}
\address{Department of Mathematics and Computer Science, Rutgers University, Newark, NJ 07102, USA}
\email{liguo@rutgers.edu}

\author[Yunnan Li]{Yunnan Li}
\address{School of Mathematics and Information Science, Guangzhou University, Waihuan Road West 230, Guangzhou 510006, China}
\email{ynli@gzhu.edu.cn}

\date{\today}

\begin{abstract}
This paper introduces the braidings of dendriform algebras and tridendriform algebras. By studying free braided dendriform algebras, we obtain braidings of the Hopf algebras of Loday and Ronco of planar binary rooted trees. We also give a variation of the braiding of Foissy for the noncommutative Connes-Kreimer (a.k.a the Foissy-Holtkamp) Hopf algebra of planar rooted forests so that the well-known isomorphism between this Hopf algebra and the Loday-Ronco Hopf algebra is extended to the braided context. As free braided tridendriform algebras, we also give braided extension of the Hopf algebra of Loday and Ronco on planar rooted trees.
\end{abstract}

\subjclass[2010]{16T05, 16T25, 16W99, 05C05} 

\keywords{braided dendriform algebra, braided tridendriform algebra, planar rooted tree, planar binary tree, braided Hopf algebra}

\maketitle

\tableofcontents

\allowdisplaybreaks

\section{Introduction}
By means of free braided dendriform algebras and tridendriform algebras, this paper constructs braided extensions of the Loday-Ronco Hopf algebra~\mcite{LR1} and Foissy-Holtkamp Hopf algebra~\mcite{Fo2,Hol} (the noncommutative Connes-Kreimer Hopf algebra), in such a way that the well-known isomorphism~\mcite{Hol} between the two Hopf algebra can be extended to the braided context.

\subsection{The motivation}
The theory of braids is important in several areas of mathematics and theoretical physics, including low dimensional topology, integrable systems and quantum field theory~\mcite{KRT,KT,Wa,YG,Oe}. In particular, braiding is intimately related to the process of quantization in both mathematics and physics, such as quantum groups~\mcite{Ka} and braid statistics in quantum mechanics~\mcite{FG}.

The Hopf algebra of rooted trees was introduced as a toy model encoding the combinatorics of Feynman graphs in the Connes-Kreimer approach of renormalization in quantum field theory~\mcite{CK}. As is well-known, the noncommutative variation of the Connes-Kreimer Hopf algebra, also known as the Foissy-Holtkamp Hopf algebra~\mcite{Fo1,Fo2,Hol}, of planar rooted trees is canonically isomorphic to another important Hopf algebra, that of Loday and Ronco on planar binary rooted trees~\mcite{LR1}. Taking the braiding approach, Foissy~\mcite{Fo3,Fo4} recently provided a quantization of the Hopf algebra of planar rooted trees. Thus it would be interesting to apply a braiding to quantize the Loday-Ronco Hopf algebra, in such as way that allows the extension of the above canonical isomorphism to the braided context. This is the motivation of our paper.

\subsection{The approach}
Our approach of braiding of the Loday-Ronco Hopf algebra starts with the observation that this Hopf algebra is the free object in the category of dendriform algebras~\mcite{Lo1}. Thus we first introduce the notion of braided dendriform algebras by providing suitable braidings on dendriform algebras. By taking the preorder of the vertices of the planar binary rooted trees, we construct the free object on the set of planar binary rooted trees decorated by a braided space, which then can be equipped with a braided Hopf algebra structure, in analogous to the Loday-Ronco Hopf algebra. We also apply a natural order of vertices to the planar rooted trees used in the Foissy-Holtkamp Hopf algebra, giving rise to a braided Hopf algebra structure on the space of planar rooted trees whose vertices are decorated by a braided space. This braiding of the Foissy-Holtkamp Hopf algebra is different from the one obtained by Foissy~\mcite{Fo1}. However, it has the benefit that the aforementioned isomorphism between the Foissy-Holtkamp Hopf algebra and the Loday-Ronco Hopf algebra can be extended to the braided context.

A second benefit of the preorder of vertices is that, for planar binary trees, it has a natural interpretation as an order of the {\it angles} of the planar binary trees which can then be extended to an order of the angles in any planar trees. This is important since angularly decorated planar rooted trees is the natural carrier of another Hopf algebra of Loday and Ronco, as the free tridendriform algebra, as well as of the free Rota-Baxter algebras~\mcite{EG2,Gu1,LR3}. Thus by extending the notions of tridendriform algebra and free tridendriform algebra to the braided context, we are also able to enrich the Loday-Ronco Hopf algebra on planar angularly decorated rooted trees to the braided context. Noting the recent work on rooted trees with decorations on edges arising from renormalization of quantum field theory~\mcite{Fo5}, we see that combining different decorations (on vertices, angles and edges) of rooted trees can reveal interesting connections on their combinatorial and algebraic structures.

\subsection{The outline}

The main body of this paper consists of three sections. In Section~\mref{sec:dend}, we continue the study of braided dendriform algebras introduced in~\mcite{GL} in connection with braided Rota-Baxter algebras, but turn our attention to the noncommutative case. In particular, the well-known Loday-Ronco~\mcite{LR1} Hopf algebra of planar binary rooted trees as the free dendriform algebra is enriched to the braided case.

In Section~\mref{sec:comp}, by a suitable ordering of the vertices of planar rooted trees, a braided enrichment is established for the Foissy-Holtkamp Hopf algebra of planar rooted trees~\mcite{Fo3,Fo4,Hol} as the noncommutative variation of the Connes-Kreimer Hopf algebra of (nonplanar) rooted trees. This ordering enables us to extend the well-known isomorphism of the Loday-Ronco Hopf algebra with the Foissy-Holtkamp Hopf algebra of planar rooted trees.

In Section~\mref{sec:trid}, we introduce the notion of a braided tridendriform algebra.
With the Loday-Ronco Hopf algebra of planar rooted trees in mind, we focus on the free objects. More precisely, we show that the Loday-Ronco Hopf algebra on planar angularly decorated rooted trees can be equipped with a braided structure and give free braided tridendriform algebras, generated by both a set and by a braided algebra. In the latter case, a subset of the braided planar (angularly decorated) rooted trees, called \lush trees are needed, possessing interesting combinatorial properties.

\medskip

\noindent
{\bf Notations.}

In this paper, we fix a ground field $\bk$ of characteristic 0. All the objects
under discussion, including vector spaces, algebras and tensor products, are taken over $\bk$.

\section{Free braided dendriform algebras and braided Loday-Ronco Hopf algebras}
\mlabel{sec:dend}

After recalling the notion of braided dendriform algebras, we construct free braided dendriform algebras on planar binary rooted trees.

\subsection{Braided dendriform algebras}

We first briefly recall the notions of braided vector spaces and algebras, while refer the reader to the literatures such as~\mcite{Ka,KT} for details.

\begin{defn}
A {\bf braiding} or {\bf Yang-Baxter operator} on a vector space $V$ is a linear
map $\sigma$ in $\mathrm{End}(V\otimes V)$ satisfying the following {\bf braid relation} on $V^{\otimes 3}$:
\begin{equation}
(\sigma\otimes
\id_{V})(\id_{V}\otimes \sigma)(\sigma\otimes
\id_{V})=(\id_{V}\otimes \sigma)(\sigma\otimes
\id_{V})(\id_{V}\otimes \sigma),
\mlabel{eq:qybe}
\end{equation}
which up to a flip is the {\bf Yang-Baxter equation} without spectral parameter.

A {\bf braided vector space}, denoted $(V,\sigma)$, is a vector space $V$ equipped with a braiding $\sigma$. For any $n\in \mathbb{N}$ and $1\leq i\leq
n-1$, we denote by $\sigma_i$ the operator $\id_V^{\otimes
	(i-1)}\otimes \sigma\otimes \id_V^{\otimes(n-i-1)}\in
\mathrm{End}(V^{\otimes n})$.

For a braided
vector space $(V,\sigma)$, a subspace $W$ of $V$ is called a {\bf braided subspace} of $(V,\sigma)$, if
\[\sigma(W\otimes V+V\otimes W)\subseteq W\otimes V+V\otimes W.\]
In this case, the quotient space $V/W$ has an induced braided vector space structure with its braiding, still denoted by $\sigma$.
\end{defn}

Denote $\mathfrak{S}_{n}$ the $n$-th symmetric group and $B_n$ the $n$-th Artin braid group; see \cite[Definition 1.1]{KT}.
For any $w\in \mathfrak{S}_{n}$, there is the corresponding lift of $w$ in $B_n$,
denoted by $T^\sigma_w$ and defined as follows: if $w=s_{i_1}\cdots s_{i_l}$ is
any reduced expression of $w$ by transpositions $s_{i_1},\cdots,s_{i_l}$, then $T^\sigma_w:=\sigma_{i_1}\cdots
\sigma_{i_l}$, uniquely determined by $\sigma$ and $w$. When $\sigma$ is the usual flip $\tau$, it reduces to the permutation action of $\mathfrak{S}_{n}$ on $V^{\otimes n}$, namely,
\[T^\tau_w(v_1\otimes\cdots\otimes v_n)=v_{w^{-1}(1)}\otimes\cdots\otimes v_{w^{-1}(n)}.\]

Using $\underline{\otimes}$ to denote the tensor product between $T(V)$ and itself to distinguish with the usual $\otimes$ within $T(V)$, we define $\beta:T(V)\underline{\otimes} T(V)\rightarrow
T(V)\underline{\otimes} T(V)$ by requiring that, for $i,j\geq 1$,  the restriction $\beta_{ij}$ of
$\beta$ to $V^{\otimes i}\underline{\otimes} V^{\otimes j}$ is $T^\sigma_{\chi_{ij}}$, where
\[\chi_{ij}:=\left(\begin{array}{cccccccc}
1&2&\cdots&i&i+1&i+2&\cdots & i+j\\
j+1&j+2&\cdots&j+i&1& 2 &\cdots & j
\end{array}\right)\in \mathfrak{S}_{i+j}.\]
By convention, $\beta_{0i}$ and $\beta_{i0}$ is the identity map on
$V^{\otimes i}$. Then $\beta$ is a braiding on $T(V)$. It is easy to verify the following equalities of $\beta$ on $T(V)$:
\begin{equation}\mlabel{beta}
\beta_{m+n,k}=(\beta_{mk}\otimes\Id_V^{\otimes n})(\Id_V^{\otimes m}\otimes\beta_{nk}),\quad
\beta_{m,n+k}=(\Id_V^{\otimes n}\otimes\beta_{mk})(\beta_{mn}\otimes\Id_V^{\otimes k}),\,m,n,k\geq0,
\end{equation}
which will be used throughout the paper. For the convenience of constructions below, we also use {\bf (un)shuffles of permutations},
$$\frakS_{i,j}:=\left\{w\in\frakS_{i+j}\,\bigg|\,{w(1)<\cdots<w(i)\atop w(i+1)<\cdots<w(i+j)}\right\}.$$
Let $\frakS_{i,\,0}:=\frakS_{0,\,i}:=\{(1)\}$ by convention.

\begin{defn}
	Let $A$ be an algebra with product $\mu$, and $\sigma$ be a braiding on $A$. We call the triple $(A,\mu,\sigma)$ a {\bf braided algebra} if it satisfies the conditions
	\begin{equation}\mlabel{ba1}
	(\mu\otimes\Id_A)\sigma_2\sigma_1=\sigma(\Id_A\otimes\mu),\quad
	(\Id_A\otimes\mu)\sigma_1\sigma_2=\sigma(\mu\otimes
	\Id_A).
	\end{equation}
	Moreover, if $A$ is unital with unit $1_A$ and satisfies
	\begin{equation}\mlabel{ba2}
	\sigma(a\otimes 1_A)=1_A\otimes a,\quad
	\sigma(1_A\otimes a)=a\otimes1_A \quad \text{for all } a\in A, \end{equation}
	then $A$ is called a {\bf unital braided algebra}.
\end{defn}	
	For braided algebras $(A,\sigma)$ and $(A',\sigma')$, a map $f:A\rightarrow A'$ is called a {\bf homomorphism of braided algebras} if
	$f$ is a homomorphism of algebras and $(f\otimes f)\sigma=\sigma'(f\otimes f)$.
	If $A$ and $A'$ are unital, then it also requires $f(1_A)=1_{A'}$.

\medskip
With motivation from periodicity in $K$-theory and making use of various categorical and operadic transitions, Loday~\mcite{Lo1} introduced the notion of a dendriform algebra.
\begin{defn}
A {\bf dendriform algebra} is a triple $(D,\prec,\succ)$ consisting of a $\bk$-module $D$ and binary operations $\prec,\succ$ on $D$ satisfying the relations
	\begin{align}
	&\mlabel{prec}(x\prec y)\prec z=x\prec(y\prec z+y\succ z),\\
	&\mlabel{ps}(x\succ y)\prec z= x\succ (y \prec z),\\
	&\mlabel{succ}(x\prec y+x\succ y)\succ z=x\succ (y\succ z) \quad \tforall x,y,z\in D.
	\end{align}
\end{defn}
Then $*:=\prec+\succ$ is associative.

One can also define the {\bf augmentation} $D^+:=\bk\oplus D$ of $D$, with $\prec$ and $\succ$ satisfying in addition
$$	1\prec x=0,\,x\prec 1=x,\,
	1\succ x=x,\,x\succ 1=0 \quad \tforall x\in D.
$$
The product $*$ is extended to $D^+$ with unit 1, while $1\prec1,\,1\succ1$ are undefined.
		
Further, a submodule $I$ of $D$ is called a {\bf dendriform ideal} if
\[D \prec I\subseteq I,\,I \prec D\subseteq I,\,D \succ I\subseteq I,\,I \succ D\subseteq I.\]
	Then the quotient module $D/I$ becomes a dendriform algebra with $\prec,\succ$ modulo $I$.
For dendriform algebras $(D,\prec,\succ)$ and $(D',\prec',\succ')$, a map $f:D\rightarrow D'$ is called a {\bf homomorphism of dendriform algebras} if
	$f$ is a linear map such that $f\prec=\prec'(f\otimes f)$ and $f\succ=\succ'(f\otimes f)$.

Combining the notions of braided algebra and dendriform algebra, we obtain the braided analogue of dendriform algebras~\mcite{GL}.
\begin{defn}
	A quadruple $(D,\prec,\succ,\sigma)$ is called a {\bf braided dendriform algebra} if $(D,\sigma)$ is a braided vector space and $(D,\prec,\succ)$ is a dendriform algebra, with the compatibility conditions
	\begin{align}\mlabel{bda1}
 \sigma(\Id_D\otimes\prec)=(\prec\otimes\Id_D)\sigma_2\sigma_1,\quad
	\sigma(\prec\otimes\Id_D)=(\Id_D\otimes\prec)\sigma_1\sigma_2,\\
	\mlabel{bda2}
	 \sigma(\Id_D\otimes\succ)=(\succ\otimes\Id_D)\sigma_2\sigma_1,\quad
	\sigma(\succ\otimes\Id_D)=(\Id_D\otimes\succ)\sigma_1\sigma_2.
	\end{align}
\end{defn}

Let $*:=\prec+\succ$. Then $(D,*,\sigma)$ is a braided algebra. Furthermore, $(D^+,*,\sigma)$ is a unital braided algebra with the braiding $\sigma$ of $D$ extended by
	\[\sigma(1\otimes x)=x\otimes 1,\,\sigma(x\otimes 1)=1\otimes x,\,x\in D^+.\]
	
Moreover, a dendriform ideal $I$ of $D$ is called a {\bf braided dendriform ideal}, if it is also a braided subspace of $D$. Then the quotient dendriform algebra $D/I$ is also a braided dendriform algebra.
For braided dendriform algebras $(D,\prec,\succ,\sigma)$ and $(D',\prec',\succ',\sigma')$, a map $f:D\rightarrow D'$ is called a {\bf homomorphism of braided dendriform algebras}, if $f$ is a homomorphism of dendriform algebras and $(f\otimes f)\sigma=\sigma'(f\otimes f)$.
	
\subsection{Free braided dendriform algebras of planar binary trees}
\mlabel{ss:freed}

We next construct free braided dendriform algebras by equipping a suitable braiding to the free dendriform algebras of planar binary trees obtained by Loday and Ronco~\mcite{LR1,Lo1} whose construction we now recall. See Section~\mref{ss:hopfd} for the Hopf algebra structure.
Let $\calb(X)$ be the set of planar binary trees with internal vertices decorated by elements in a set $X$. Some example are
\[|\,,\,\raisebox{-.1mm}{\xy 0;/r.3pc/:
	(1.5,4)*{};(0,2)*{}**\dir{-};
	(-1.5,4)*{};(0,2)*{}**\dir{-};
	(0,0)*{};(0,2)*{}**\dir{-};(1.5,1.5)*{\scriptstyle x};
	\endxy}\,,\,\raisebox{-.1mm}{\xy 0;/r.3pc/:
	(3,6)*{};(0,2)*{}**\dir{-};
	(-3,6)*{};(0,2)*{}**\dir{-};
	(-1.5,4)*{};(0,6)*{}**\dir{-};
	(0,0)*{};(0,2)*{}**\dir{-};
	(1.5,1.5)*{\scriptstyle y};
	(-3,3.5)*{\scriptstyle x};
	\endxy}\,,\,\raisebox{-.1mm}{\xy 0;/r.3pc/:
	(3,6)*{};(0,2)*{}**\dir{-};
	(-3,6)*{};(0,2)*{}**\dir{-};
	(1.5,4)*{};(0,6)*{}**\dir{-};
	(0,0)*{};(0,2)*{}**\dir{-};
	(1.5,1.5)*{\scriptstyle x};
	(3,3.5)*{\scriptstyle y};
	\endxy}\,,\,\raisebox{-.1mm}{\xy 0;/r.3pc/:
	(3,6)*{};(0,2)*{}**\dir{-};
	(-3,6)*{};(0,2)*{}**\dir{-};
	(1.5,4)*{};(.5,6)*{}**\dir{-};
	(-1.5,4)*{};(-.5,6)*{}**\dir{-};
	(0,0)*{};(0,2)*{}**\dir{-};
	(1.5,1.5)*{\scriptstyle y};
	(-3,3.5)*{\scriptstyle x};
	(3,3.5)*{\scriptstyle z};
	\endxy}\,,\dots\]
for $x,y,z\in X$. When $X$ is a singleton, we abbreviate $\calb(X)$ as $\calb$, for planar binary trees without decorations. Define the {\bf (vertex) degree} $d_v(Y)$ of any $Y\in \calb$ to be the number of its internal vertices and, for $n\geq 0$, let $\calb_n$ be the subset of $\calb$ consisting of planar binary trees with $n$ internal vertices. It is well-known that $|\calb_n|$ is the $n$-th Catalan number $C_n$.

To obtain a braided structure, we fix the following natural well-order on the set of internal vertices of any $Y\in\calb_n,\,n\geq1$. Label the $n+1$ leaves in $Y$ from left to right by $0,1,\dots,n$, and then let the internal vertex located between the $i$-th and $(i+1)$-th leaves be the $(i+1)$-th vertex for $0\leq i\leq n-1$. We call this the {\bf canonical order} on the internal vertices of $Y$. Let $Y(x_1,\dots,x_n)\in\calb_n(X)$ denote the binary tree $Y$ with its canonically ordered internal vertices decorated by $x_1,\dots,x_n\in X$ from  left to right. This way we have the identification
\begin{equation}
\calb_n (X)= \calb_n\times X^n,\, Y(x_1,\cdots,x_n) \leftrightarrow (Y,(x_1,\cdots,x_n)), \quad Y\in \calb_n, x_1,\cdots, x_n\in X,
\mlabel{eq:yset}
\end{equation}
with which we can write
$\bfk \calb(X)=\bigoplus_{n\geq 1} \bfk (B_n\times X^n)$.
\begin{defn}
Given a vector space $V$, define
\[\caly(V):=\bigoplus_{n\geq0}\caly_n(V):=\bigoplus_{n\geq0}\bk\calb_n\otimes V^{\otimes n}.\]
In particular, $\caly_0(V)=\bfk\,|$ by convention. Denote $\oline{\caly(V)}:=\bigoplus_{n\geq1}\caly_n(V)$ as a subspace of $\caly(V)$.
\end{defn}
In analogous to Eq.~\meqref{eq:yset}, we can identify a pure tensor $Y\otimes (v_1\ot \cdots \ot v_n)$, where $Y\in \calb_n$ and $v_1\ot \cdots \ot v_n\in V^{\ot n}$, with the decorated tree $Y(v_1,\cdots,v_n)$ in which the binary tree $Y$ has its canonically ordered internal vertices decorated by $v_1,\cdots, v_n$ from left to right, allowing $\bfk$-linearity in each decoration and hence multi-linearity in all the decorations.
Also, denote
\[\bar{\calb}_n(V):=\{Y(v_1,\dots,v_n)\in\caly(V)\,|\,Y\in\calb_n,v_1,\dots,v_n\in V\},\,n\geq0,\]
and $\bar{\calb}(V):=\sqcup_{n\geq0}\bar{\calb}_n(V)$.
Define the ($\bk$-linear) grafting operation on $\caly(V)$,
\[\vee:\caly(V)\otimes V\otimes \caly(V)\rar\caly(V),\]
such that $\vee(Y\otimes v\otimes Y')$ is obtained by using $\raisebox{-.1mm}{\xy 0;/r.3pc/:
	(1.5,4)*{};(0,2)*{}**\dir{-};
	(-1.5,4)*{};(0,2)*{}**\dir{-};
	(0,0)*{};(0,2)*{}**\dir{-};(1.5,1.5)*{\scriptstyle v};
	\endxy}$
to graft two binary trees $Y$ and $Y'$ into one, and also abbreviated as $Y\vee_v Y'$.
Then define binary operations $\prec$ and $\succ$ on $\caly(V)$ by the recursion
\begin{align}
&Y\prec |:=Y,\quad |\prec Y:=0,\quad Y\succ |:=0,\quad |\succ Y:=Y,
\mlabel{eq:dendi}\\
&Y\prec Y':=Y_1\vee_u(Y_2*Y'), \quad Y\succ Y':=(Y*Y_1')\vee_v Y_2',
\mlabel{eq:dendii}
\end{align}
for any $Y=Y_1\vee_u Y_2$ and $Y'=Y_1'\vee_v Y_2'$ in $\bar{\calb}(V)$, where $*_\sigma:=\prec_\sigma + \succ_\sigma$ on $\oline{\caly(V)}$, and is extended to $\caly(V)$ with $|$\, being the unit.
By~\cite[Theorem 5.8]{Lo1}, the space $\oline{\caly(V)}$ with operations $\prec,\succ$ defined in Eq.~\meqref{eq:dendii} is the free dendriform algebra on $V$. Thus $(\caly(V),*)$ is an associative algebra.

\smallskip
Now suppose that $(V,\sigma)$ is moreover a braided vector space,
the braiding $\sigma$ on $V$ induces a braiding $\sigma_{BT}$ on $\caly(V)$ defined by
\begin{align*}
&\sigma_{BT}(Y(v_1,\dots,v_m)\otimes|\,)=|\otimes Y(v_1,\dots,v_m),\,
\sigma_{BT}(\,|\otimes Y(v_1,\dots,v_m))=Y(v_1,\dots,v_m)\otimes|\,,\\
&\sigma_{BT}(Y(v_1,\dots,v_m)\otimes Y'(v_{m+1},\dots,v_{m+n}))=(Y'\otimes Y)\beta_{mn}(v_1\otimes\cdots\otimes v_{m+n}),
\end{align*}
for any $Y\in\calb_m,Y'\in\calb_n$ and $v_1,\dots,v_{m+n}\in V$. Also, we define the $\bk$-linear map
\[\sigma_{V,\caly(V)}:V\otimes\caly(V)\rar \caly(V)\otimes V,\,v\otimes Y(v_1,\dots,v_m)\mapsto (Y\otimes\Id_V)\beta_{1,\,m}(v\otimes v_1\otimes\cdots\otimes v_m)\]
for any $Y\in\calb_m$ and $v,v_1,\dots,v_m\in V$. By symmetry, we define $\sigma_{\caly(V),V}:\caly(V)\otimes V\rar V\otimes\caly(V)$.
Then by the definitions of operator $\vee$ and braidings $\sigma_{BT},\,\sigma_{V,\caly(V)}$ and $\sigma_{\caly(V),V}$, we have
\begin{lemma}
For any $v\in V$ and homogeneous $Y,Y',Y''\in\bar{\calb}(V)$, we have
\begin{align}
&\mlabel{tv1}\sigma_{BT}((Y\vee_v Y')\otimes Y'')=(\Id_{\caly(V)}\otimes\vee)(\sigma_{BT})_1(\sigma_{V,\caly(V)})_2(\sigma_{BT})_3(Y\otimes v\otimes Y'\otimes Y''),\\
&\mlabel{tv2}\sigma_{BT}(Y\otimes(Y'\vee_v Y''))=(\vee\otimes\Id_{\caly(V)})(\sigma_{BT})_3(\sigma_{\caly(V),V})_2(\sigma_{BT})_1(Y\otimes Y'\otimes v\otimes Y'').
\end{align}	
\end{lemma}

To obtain a braided dendriform algebra structure on $\caly(V)$ for a braided vector space $(V,\sigma)$, we rewrite $\prec,\succ,*$ on $\caly(V)$ given in \meqref{eq:dendi}--\meqref{eq:dendii} with enriched notations $\prec_\sigma,\succ_\sigma,*_\sigma$, and show that it is compatible with the braiding structure induced from $(V,\sigma)$.
\begin{theorem}
Given a braided vector space $(V,\sigma)$,
the quadruple $(\oline{\caly(V)},\prec_\sigma,\succ_\sigma,\sigma_{BT})$ is a braided dendriform algebra.
\mlabel{bdbt}
\end{theorem}
\begin{proof} 	
First fix $Y(v_1,\dots,v_m)\in\bar{\calb}_m(V)$, $Y'(v_{m+1},\dots,v_{m+n})\in\bar{\calb}_n(V)$ and $Y''(v_{m+n+1},\dots,v_{m+n+l})\in\bar{\calb}_l(V)$ with $m,n,l\geq0$.
As noted above, $(\oline{\caly(V)},\prec_\sigma,\succ_\sigma)$ is the dendriform algebra in~\cite[Theorem 5.8]{Lo1}.
So it remains to check the compatibility conditions in Eqs.~\meqref{bda1} and \meqref{bda2} together with Eqs.~\meqref{ba1} and \meqref{ba2} for $\caly(V)$, for which we apply induction on the vertex degree $d_v$. If one of $Y,Y',Y''$ is $|$\,, then the verification is clear. If not, suppose $m, n, l\geq1$, and $Y=Y_1\vee_{v_i}Y_2$ with $1\leq i\leq m$, then
\[\begin{split}
\sigma&_{BT}((Y\prec_\sigma Y')\otimes Y'')(v_1\otimes\cdots\otimes v_{m+n+l})
=\sigma_{BT}((Y_1\vee_{v_i}(Y_2*_\sigma Y'))\otimes Y'')\\
&=(\Id_{\caly(V)}\otimes\vee)(\sigma_{BT})_1(\sigma_{V,\caly(V)})_2(\sigma_{BT})_3(Y_1\otimes v_i\otimes (Y_2*_\sigma Y') \otimes Y'')\\
&=(\Id_{\caly(V)}\otimes\vee(\Id_{\caly(V)}\otimes\Id_V\otimes*_\sigma))(\sigma_{BT})_1(\sigma_{V,\caly(V)})_2(\sigma_{BT})_3(\sigma_{BT})_4(Y_1\otimes v_i\otimes Y_2\otimes Y' \otimes Y'')\\
&=(\Id_{\caly(V)}\otimes\prec_\sigma)(\sigma_{BT})_1(\sigma_{BT})_2(Y\otimes Y' \otimes Y''),
\end{split}\]
where the second equality is due to Eq.~\meqref{tv1} suppressing the decorations $v_1,\dots,v_{m+n+l}$ by convention; the third one holds by the induction hypothesis for Eq.~\meqref{ba1} of $*_\sigma$, and the last one is by Eq.~\meqref{beta}.
The proof of
\[\sigma_{BT}(Y\otimes (Y'\prec_\sigma Y''))=(\prec_\sigma\otimes \Id_{\caly(V)})(\sigma_{BT})_2(\sigma_{BT})_1(Y\otimes Y'\otimes Y'')\]
is similar. Further, for $Y'=Y'_1\vee_{v_j}Y'_2$ with $m+1\leq j\leq m+n$, we have
\begin{align*}
\sigma&_{BT}((Y\succ_\sigma Y')\otimes Y'')(v_1\otimes\cdots\otimes v_{m+n+l})=\sigma_{BT}(((Y*_\sigma Y_1')\vee_{v_j} Y_2')\otimes Y'')\\
&=(\Id_{\caly(V)}\otimes\vee)(\sigma_{BT})_1(\sigma_{V,\caly(V)})_2(\sigma_{BT})_3((Y*_\sigma Y_1')\otimes v_j\otimes Y_2'\otimes Y'')\\
&=(\Id_{\caly(V)}\otimes
(\vee(*_\sigma\otimes\Id_V\otimes \Id_{\caly(V)})))(\sigma_{BT})_1(\sigma_{BT})_2(\sigma_{V,\caly(V)})_3(\sigma_{BT})_4(Y\otimes Y_1'\otimes v_j\otimes Y_2'\otimes Y'')\\
&=(\Id_{\caly(V)}\otimes\succ_\sigma)(\sigma_{BT})_1(\sigma_{BT})_2(Y\otimes Y'\otimes Y''),
\end{align*}
where the second equality is due to Eq.~\meqref{tv1}; the third one holds by the induction hypothesis for Eq.~\meqref{ba1} on $*_\sigma$ and the last one is by Eq.~\meqref{beta}.
A similar argument gives
\[\sigma_{BT}(Y\otimes (Y'\succ_\sigma Y''))=(\succ_\sigma\otimes \Id_{\caly(V)})(\sigma_{BT})_2(\sigma_{BT})_1(Y\otimes Y'\otimes Y'').\]
So $(\oline{\caly(V)},\prec_\sigma,\succ_\sigma,\sigma_{BT})$ is a braided dendriform algebra and $(\caly(V),*_\sigma,\sigma_{BT})$ a braided algebra.
\end{proof}

Since $\sigma_{BT}$ is homogeneous, we can regard $\oline{\caly(V)}$ as a braided subspace of $\caly(V)$. Let
\[i_V:V\rar \oline{\caly(V)},\,v\mapsto\raisebox{-.1mm}{\xy 0;/r.3pc/:
	(1.5,4)*{};(0,2)*{}**\dir{-};
	(-1.5,4)*{};(0,2)*{}**\dir{-};
	(0,0)*{};(0,2)*{}**\dir{-};(1.5,1.5)*{\scriptstyle v};
	\endxy}\]
be the natural embedding of braided vector spaces. Denote
$\raisebox{-.1mm}{\xy 0;/r.3pc/:
	(1.5,4)*{};(0,2)*{}**\dir{-};
	(-1.5,4)*{};(0,2)*{}**\dir{-};
	(0,0)*{};(0,2)*{}**\dir{-};(1.5,1.5)*{\scriptstyle v};
	\endxy}=|\vee_v|$ by $Y[v]$ for $v\in V$.

Now we are in the position to show that $\oline{\caly(V)}$ is a free object in the category of braided dendriform algebras, as a braided version of Loday's result~\cite[Theorem 5.8]{Lo1}.

\begin{theorem}
Given a braided vector space $(V,\sigma)$, the quadruple $(\oline{\caly(V)},\prec_\sigma,\succ_\sigma,\sigma_{BT})$ together with the map $i_V: V\rar \oline{\caly(V)}$ is the free braided dendriform algebra on $(V,\sigma)$. More precisely, for any braided dendriform algebra $(D,\prec_D,\succ_D,\tau)$ and a homomorphism $\varphi:V\rar D$ of braided vector spaces, there exists a unique homomorphism $\bar{\varphi}:\oline{\caly(V)}\rar D$ of braided dendriform algebras such that $\varphi=\bar{\varphi}\circ i_V$.
\mlabel{fbd}
\end{theorem}
\begin{proof}
Given any braided dendriform algebra $(D,\prec_D,\succ_D,\tau)$ and a homomorphism $\varphi:V\rar D$ of braided vector spaces, one can define a linear map $\bar{\varphi}:\caly(V)\rar D^+:=D\oplus \bfk 1$ recursively by
\begin{align*}
&\bar{\varphi}(\,|\,)=1,\,\bar{\varphi}(Y[v])=\varphi(v),\\
&\bar{\varphi}(Y\vee_v Y')=\bar{\varphi}(Y)\succ_D\varphi(v)\prec_D\bar{\varphi}(Y'),
\end{align*}
for $Y(v_1,\dots,v_m)$ and $Y'(v_{m+1},\dots,v_{m+n})$ with $m,n\geq0$ and $v_1,\dots,v_{m+n},v\in V$.
Checking as in the proof of \cite[Theorem 5.8]{Lo1}, we find that the restriction of $\bar{\varphi}$ to $\oline{\caly(V)}$ is a homomorphism of dendriform algebras. Next we show that
\[(\bar{\varphi}\otimes\bar{\varphi})\sigma_{BT}(Y\otimes Y')=\tau(\bar{\varphi}\otimes\bar{\varphi})(Y\otimes Y')\]
by induction on the vertex degree $d_v$ of $Y$. If $Y=|\,$, it is clear by the definition of $\sigma_{BT}$.
Otherwise, one can check it in four cases depending on whether $Y_1$ or $Y_2$ is $|$ or not for $Y=Y_1\vee_{v_i} Y_2$ with $1\leq i\leq m$. Here we only consider the case when $Y_1,\,Y_2\neq|$ since other cases are easier to verify.
\begin{align*}
(\bar{\varphi}&\otimes\bar{\varphi})\sigma_{BT}(Y\otimes Y')=(\bar{\varphi}\otimes\bar{\varphi})(Y'\otimes Y)\beta_{mn}(v_1\otimes \cdots\otimes v_{m+n})\\
&=(\bar{\varphi}\otimes(\bar{\varphi}\succ_D\varphi\prec_D\bar{\varphi}))(Y'\otimes Y_1\otimes\Id_V\otimes Y_2)\beta_{mn}(v_1\otimes \cdots\otimes v_{m+n})\\
&=(\bar{\varphi}\otimes(\bar{\varphi}\succ_D\varphi\prec_D\bar{\varphi}))(\sigma_{BT})_1(\sigma_{V,\caly(V)})_2(\sigma_{BT})_3(Y_1\otimes v_i\otimes Y_2\otimes Y')\\
&=(\Id_D\otimes\succ_D\prec_D)\tau_1\tau_2\tau_3(\bar{\varphi}\otimes\varphi\otimes\bar{\varphi}\otimes\bar{\varphi})(Y_1\otimes v_i\otimes Y_2\otimes Y')\\
&=\tau((\bar{\varphi}\succ_D\varphi\prec_D\bar{\varphi})\otimes\bar{\varphi})(Y_1\otimes v_i\otimes Y_2\otimes Y')=\tau(\bar{\varphi}\otimes\bar{\varphi})(Y\otimes Y'),
\end{align*}
where the first and third equalities are by the definitions of $\sigma_{BT}$ and $\sigma_{V,\caly(V)}$; the second and last equalities are by the definition of $\bar{\varphi}$; the fourth equality uses the induction hypothesis; and the fifth one is due to Eqs.~\meqref{bda1} and \meqref{bda2} on $D$. Hence, we have $(\bar{\varphi}\otimes\bar{\varphi})\sigma_{BT}=\tau(\bar{\varphi}\otimes\bar{\varphi})$.

By definition, we know that $\bar{\varphi}$ is uniquely determined by its image on $Y[v]$ as a homomorphism of dendriform algebras with $\bar{\varphi}(\,|\,)=1$, while $\varphi=\bar{\varphi}\circ i_V$ means that $\varphi(v)=\bar{\varphi}(Y[v])$ for any $v\in V$. Hence, $(\oline{\caly(V)},\prec_\sigma,\succ_\sigma,\sigma_{BT})$ with $i_V:V\rar\oline{\caly(V)}$ has the desired universal property.
\end{proof}

\subsection{The braided Hopf algebra of planar binary trees}
\mlabel{ss:hopfd}

Dual to the notion of bradied algebras, we introduce the notion of braided coalgebras, and then braided bialgebras and braided Hopf alebras. See \mcite{Ka,Mon} for details.
\begin{defn}
	We call $(C,\Delta,\varepsilon,\sigma)$ a {\bf braided coalgebra} if $(C,\Delta,\varepsilon)$ is a coalgebra with braiding $\sigma$ and satisfies
	\begin{align}
	&\mlabel{bc1}
\sigma_1\sigma_2(\Delta\otimes\Id_C)=(\Id_C\otimes\Delta)\sigma,\,
	\sigma_2\sigma_1(\Id_C\otimes\Delta)=(\Delta\otimes
	\Id_C)\sigma,\\
    &\mlabel{bc2}
	 (\varepsilon\otimes\Id_C)\sigma=\Id_C=(\Id_C\otimes\varepsilon)\sigma.
	\end{align}
	Also, a quintuple $(H,\mu,\Delta,\varepsilon,\sigma)$ is called a {\bf braided bialgebra}, if $(H,\mu,\sigma)$ (resp. $(H,\Delta,\varepsilon,\sigma)$) is a braided algebra (resp. coalgebra) such that
	\begin{equation}
	\Delta\,\mu=(\mu\otimes\mu)\sigma_2(\Delta\otimes \Delta).
	\mlabel{eq:comp}
	\end{equation}
	When $H$ has unit $1_H$ and antipode $S:H\to H$ such that $\mu(S\otimes\Id_H)\Delta=\varepsilon 1_H=\mu(\Id_H\otimes S)\Delta$,
	then the septuple $(H,\mu,\Delta,\vep,S,\sigma)$ is a {\bf braided Hopf algebra}. The homomorphisms for these braided objects are similarly defined as in the case of braided algebras.
\end{defn}

We also introduce the following useful notion~\cite[\S 5.2]{Mon}.
\begin{defn}
	For any coalgebra $(C,\Delta,\varepsilon)$, let $C_0$ be the sum of all simple subcoalgebras of $C$, called the {\bf coradical} of $C$, and recursively define
	\[C_n=\{x\in C\,|\,\Delta(x)\in C_0\otimes C+C\otimes C_{n-1}\},\,n\geq1.\]
	Then $\{C_n\}_{n\geq0}$ is a subcoalgebra filtration of $C$, called the {\bf coradical filtration} of $C$.
	In particular, if the coradical $C_0$ of $C$ is one-dimensional, then $C$ is called {\bf connected}.
\end{defn}

We recall the following well-known fact of Takeuchi~\cite[Lemma 5.2.10]{Mon}. See also~\mcite{Kh,LR3,Ta}.
\begin{lemma}
	Any (braided) bialgebra connected as a coalgebra is a (braided) Hopf algebra.
\mlabel{cbch}
\end{lemma}

Now we recall the well-known Loday-Ronco Hopf algebra~\mcite{LR1}.
\begin{defn}
The {\bf Loday-Ronco Hopf algebra} $\calh_{BT}(X)$ of planar binary trees is the $\bk$-linear space $\bk\calb(X)=\bigoplus_{n\geq0}\bk\calb_n(X)$, in which
\begin{enumerate}
\item
the product $*$ is recursively defined by
\[|*Y=Y*|=Y,\quad Y*Y'=Y_1\vee_x(Y_2*Y')+(Y*Y_1')\vee_y Y_2',\]
for any $Y=Y_1\vee_x Y_2,\,Y'=Y_1'\vee_y Y_2',\,x,y\in X$, and with the unit $\,|\,$,
\item
the coproduct $\Delta_{BT}$ is recursively defined by
\begin{equation}\mlabel{lrco}
\Delta_{BT}(\,|\,)=|\otimes |\,,\quad\,\Delta_{BT}(Y\vee_x Y')=(Y\vee_x Y')\otimes |+ \sum (Y_{(1)}*Y'_{(1)})\otimes(Y_{(2)}\vee_x Y'_{(2)}),
\end{equation}
where $\Delta_{BT}(Y)=\sum Y_{(1)}\otimes Y_{(2)},\, \Delta_{BT}(Y')=\sum Y'_{(1)}\otimes Y'_{(2)}$, and the counit $\varepsilon_{BT}$ has $\varepsilon_{BT}(Y)=\delta_{Y,\,|}$.
\end{enumerate}
\mlabel{lr}
\end{defn}

The coproduct $\Delta_{BT}$ on $\calh_{BT}(X)$ can be linearly extended to $\caly(V)$ as mentioned in \cite[\S 5.13]{Lo1}. Next we define $\bk$-linear map
$\Delta'_{BT}:\caly(V)\rar\caly(V)\otimes\caly(V)$
recursively by
\[\Delta'_{BT}(\,|\,)=|\otimes |\,,\,
\Delta'_{BT}(Y\vee_v Y')=(Y\vee_v Y')\otimes |+(*_\sigma,\vee)(\Delta'_{BT}(Y)\otimes v\otimes\Delta'_{BT}(Y')),\]
for any $Y,Y'\in\bar{\calb}(V),\,v\in V$, where $ (*_\sigma,\vee):=(*_\sigma\otimes\vee)(\sigma_{BT})_2(\sigma_{V,\caly(V)})_3$. Then $\Delta'_{BT}$ is a braided version of $\Delta_{BT}$.
Define a linear map $\varepsilon'_{BT}:\caly(V)\rar\bk$ by $\varepsilon'_{BT}(Y)=\delta_{Y,\,|}$.

\begin{theorem}
	Given a braided vector space $(V,\sigma)$,
	the quintuple $(\caly(V),*_\sigma,\Delta'_{BT},\varepsilon'_{BT},\sigma_{BT})$
	is a braided Hopf algebra.
\mlabel{bhbt}
\end{theorem}
\begin{proof}
The fact that $(\caly(V),*_\sigma,\sigma_{BT})$ is a braided algebra is checked in Theorem \mref{bdbt}. Next we show that $(\caly(V),*_\sigma,\Delta'_{BT},\sigma_{BT})$ is a braided bialgebra.

Let us fix $Y(v_1,\dots,v_m)=Y_1\vee_{v_i} Y_2$ and $Y'(v_{m+1},\dots,v_{m+n})=Y_1'\vee_{v_j} Y_2'$, $,\,1\leq i\leq m,\,m+1\leq j\leq m+n$. First we check the compatibility in Eq.~\meqref{eq:comp} between $*_\sigma$ and $\Delta'_{BT}$ for which we apply induction on the vertex degree $d_v$. When one of $Y,Y'$ is $|$\,, it is clear. Otherwise, with the notation
$(*_\sigma,*_\sigma)=(*_\sigma\otimes *_\sigma)(\sigma_{BT})_2$,
$(*_\sigma,\prec_\sigma)=(*_\sigma\otimes\prec_\sigma)(\sigma_{BT})_2$ and $(*_\sigma,\succ_\sigma)=(*_\sigma\otimes\succ_\sigma)(\sigma_{BT})_2$, we have
\begin{align*}
\Delta'_{BT}&(Y*_\sigma Y')=\Delta'_{BT}(Y_1\vee_{v_i}(Y_2*_\sigma Y')+(Y*_\sigma Y_1')\vee_{v_j} Y_2')\\
&=(Y*_\sigma Y')\otimes |
+(*_\sigma,\vee)(\Delta'_{BT}(Y_1)\otimes v_i\otimes\Delta'_{BT}(Y_2*_\sigma Y'))\\
&\quad+(*_\sigma,\vee)(\Delta'_{BT}(Y*_\sigma Y_1')\otimes v_j\otimes\Delta'_{BT}(Y_2'))\\
&=(Y*_\sigma Y')\otimes |+(*_\sigma,\prec_\sigma)\left(\Delta'_{BT}(Y_1\vee_{v_i} Y_2)\otimes\Delta'_{BT}(Y')-(Y_1\vee_{v_i} Y_2)\otimes |\otimes Y'\otimes |\,\right)\\
&\quad+(*_\sigma,\succ_\sigma)\left(\Delta'_{BT}(Y)\otimes\Delta'_{BT}(Y'_1\vee_{v_j} Y'_2)-Y\otimes |\otimes (Y'_1\vee_{v_j} Y'_2)\otimes |\,\right)\\
&=(*_\sigma,*_\sigma)(\Delta'_{BT}(Y)\otimes \Delta'_{BT}(Y')),
\end{align*}
where the first and second equalities are by the definitions of $*_\sigma$ and $\Delta'_{BT}$ respectively; the third one is due to the induction hypothesis and the definitions of $\prec_\sigma,\,\succ_\sigma$.

Next we show that  $(\caly(V),\Delta'_{BT},\varepsilon'_{BT},\sigma_{BT})$ is a braided coalgebra. Then $(\caly(V),*_\sigma,\Delta'_{BT},\varepsilon'_{BT},\sigma_{BT})$ is a braided bialgebra by previously proved relation. We first need to check the conditions in Eqs.~\meqref{bc1} and \meqref{bc2} for $\Delta'_{BT}$, but Eq.~\meqref{bc2} is clear by definition. For condition \meqref{bc1}, we show it by induction on $d_v$. Again, if one of $Y,Y'$ is $|$\,, it is obvious. If not, then
\begin{align*}
(\Delta'_{BT}&\otimes\Id_{\caly(V)})\sigma_{BT}(Y\otimes Y')
=(\sigma_{BT})_2(\sigma_{BT})_1(Y\otimes Y'\otimes |\,)\\ &\quad +((*_\sigma,\vee)(\Delta'_{BT}\otimes\Id_V\otimes\Delta'_{BT})\otimes\Id_{\caly(V)})(\sigma_{BT})_3(\sigma_{\caly(V),V})_2(\sigma_{BT})_1(Y\otimes Y'_1\otimes v_j \otimes Y'_2)\\
&=(\sigma_{BT})_2(\sigma_{BT})_1(Y\otimes Y'\otimes |\,)+(*_\sigma\otimes\vee\otimes\Id_{\caly(V)})(\sigma_{BT})_2(\sigma_{V,\caly(V)})_3\\
&\quad\,(\Delta'_{BT}\otimes\Id_V\otimes\Delta'_{BT}\otimes\Id_{\caly(V)})(\sigma_{BT})_3(\sigma_{\caly(V),V})_2(\sigma_{BT})_1(Y\otimes Y'_1\otimes v_j \otimes Y'_2)\\
&=(\sigma_{BT})_2(\sigma_{BT})_1(Y\otimes Y'\otimes |\,)+(*_\sigma\otimes\vee\otimes\Id_{\caly(V)})(\sigma_{BT})_5(\sigma_{\caly(V),V})_4
(\sigma_{BT})_3(\sigma_{BT})_2(\sigma_{BT})_1\\
&\quad\,(\sigma_{BT})_3(\sigma_{V,\caly(V)})_4(Y\otimes\Delta'_{BT}(Y'_1)\otimes v_j \otimes\Delta'_{BT}(Y'_2))\\
&=(\sigma_{BT})_2(\sigma_{BT})_1(Y\otimes Y'\otimes |\,)+(\sigma_{BT})_2(\sigma_{BT})_1(\Id_{\caly(V)}\otimes(*_\sigma,\vee))(Y\otimes\Delta'_{BT}(Y'_1)\otimes v_j \otimes\Delta'_{BT}(Y'_2))\\
&=(\sigma_{BT})_2(\sigma_{BT})_1(\Id_{\caly(V)}\otimes\Delta'_{BT})(Y\otimes Y'),
\end{align*}
where the first and last equalities are by the definition of $\Delta'_{BT}$; the second one is by the definition of $(*_\sigma,\vee)$; the third one uses the induction hypothesis; the fourth one is due to conditions in Eq.~\meqref{ba1} of $*_\sigma$ and in Eq.~\meqref{tv1} of $\vee$. The second identity in Eq.~\meqref{bc1} for $\Delta'_{BT}$ is similar to check.

For the coassociativity of $\Delta'_{BT}$, it is also proved by induction on $d_v$. In fact, given $u\in V$,
\begin{align*}
(\Delta'_{BT}&\otimes\Id_{\caly(V)})\Delta'_{BT}(Y\vee_u Y')=\Delta'_{BT}(Y\vee_u Y')\otimes |+(\Delta'_{BT}\otimes\Id_{\caly(V)})
(*_\sigma,\vee)(\Delta'_{BT}(Y)\otimes u\otimes\Delta'_{BT}(Y'))\\
&=(Y\vee_u Y')\otimes |\otimes |+ (*_\sigma,\vee)(\Delta'_{BT}(Y)\otimes u\otimes\Delta'_{BT}(Y'))\otimes |\\
&\quad +(\Delta'_{BT}\otimes\Id_{\caly(V)})(*_\sigma\otimes \vee)(\sigma_{BT})_2(\sigma_{V,\caly(V)})_3(\Delta'_{BT}(Y)\otimes u\otimes\Delta'_{BT}(Y'))\\
&=(Y\vee_u Y')\otimes |\otimes |+ (*_\sigma,\vee)(\Delta'_{BT}(Y)\otimes u\otimes\Delta'_{BT}(Y'))\otimes |\\
&\quad +(*_\sigma\otimes*_\sigma\otimes\Id_{\caly(V)})(\sigma_{BT})_2(\Delta'_{BT}\otimes \Delta'_{BT}\otimes \vee)(\sigma_{BT})_2(\sigma_{V,\caly(V)})_3(\Delta'_{BT}(Y)\otimes u\otimes\Delta'_{BT}(Y'))\\
&=(Y\vee_u Y')\otimes |\otimes |+ (*_\sigma,\vee)(\Delta'_{BT}(Y)\otimes u\otimes\Delta'_{BT}(Y'))\otimes |\\
&\quad +(*_\sigma\otimes(*_\sigma,\vee))(\sigma_{BT})_2(\sigma_{BT})_3(\sigma_{V,\caly(V)})_4((\Id_{\caly(V)}\otimes\Delta'_{BT})\Delta'_{BT}(Y)\otimes u\otimes(\Id_{\caly(V)}\otimes\Delta'_{BT})\Delta'_{BT}(Y'))\\
&=(Y\vee_u Y')\otimes |\otimes |+ (*_\sigma,\vee)(\Delta'_{BT}(Y)\otimes u\otimes\Delta'_{BT}(Y'))\otimes |\\
&\quad +(*_\sigma\otimes(*_\sigma,\vee)(\Delta'_{BT}\otimes \Id_V\otimes\Delta'_{BT}))
(\sigma_{BT})_2(\sigma_{V,\caly(V)})_3(\Delta'_{BT}(Y)\otimes u\otimes\Delta'_{BT}(Y'))\\
&=(Y\vee_u Y')\otimes\Delta'_{BT}(\,|\,)+(\Id_{\caly(V)}\otimes\Delta'_{BT})
(*_\sigma,\vee)(\Delta'_{BT}(Y)\otimes u\otimes\Delta'_{BT}(Y'))\\
&=(\Id_{\caly(V)}\otimes\Delta'_{BT})\Delta'_{BT}(Y\vee_u Y'),
\end{align*}
where the first, second, sixth and last equalities are by the definition of $\Delta'_{BT}$; the third one uses the compatibility between $\Delta'_{BT}$ and $*_\sigma$; the fourth one is
due to Eq.~\meqref{bc1} of $\Delta'_{BT}$ and the induction hypothesis; the fifth one is also by Eq.~\meqref{bc1}.

Since the braided bialgebra $(\caly(V),*_\sigma,\Delta'_{BT},\sigma_{BT})$ is
connected graded with respect to the degree $d_v$, it is a braided Hopf algebra by Lemma~\mref{cbch}. See also~\cite[Lemma 1.27]{Fo4} or \cite[Lemma 12]{Fo3}.
\end{proof}

Recently, Gao and Zhang~\mcite{GZ1} provided an explicit combinatorial description of $\Delta_{BT}$ via certain cuts on decorated planar binary trees. We recall their construction and show that it can be extended to the braided case.

For any $Y\in\calb(X)$, a {\bf subforest} $P$ of $Y$ is defined to be the concatenation of planar binary trees consisting of a subset of internal vertices of $Y$ together with their descents and edges containing all these involved vertices.

For a subforest $P$ of $Y$, let $P^*\in\bk\calb(X)$ denote the result tree when the concatenation product of trees in $P$ is replace by the product $*$, and let $Y/P\in\calb(X)$ be obtained by removing $P$ from $Y$ in the following way. Divide all edges connecting $P$ to the rest of $Y$ into two edges such that one belongs to $P$ and the other one belonging to the rest of $Y$. Here the decorations on vertices of $Y$ remain at their corresponding positions in $P$ and $Y/P$. Let $\calf_Y$ be the set of subforests of $Y$, including the empty tree $\emptyset$ and the whole tree $Y$. With the above notation, it is shown in~\mcite{GZ1} that
\begin{equation}
\Delta_{BT}(Y)=\sum_{P\in\calf_Y}P^*\otimes(Y/P),\,Y\in\calb(X).
\mlabel{lrco1}
\end{equation}

For example, if $Y=\raisebox{-.1mm}{\xy 0;/r.3pc/:
	(3,6)*{};(0,2)*{}**\dir{-};
	(-3,6)*{};(0,2)*{}**\dir{-};
	(1.5,4)*{};(.5,6)*{}**\dir{-};
	(-1.5,4)*{};(-.5,6)*{}**\dir{-};
	(0,0)*{};(0,2)*{}**\dir{-};
	(1.5,1.5)*{\scriptstyle 2};
	(-3,3.5)*{\scriptstyle 1};
	(3,3.5)*{\scriptstyle 3};
	\endxy}$, then
\[\begin{split}
\Delta_{BT}(Y)&=\raisebox{-.1mm}{\xy 0;/r.3pc/:
	(3,6)*{};(0,2)*{}**\dir{-};
	(-3,6)*{};(0,2)*{}**\dir{-};
	(1.5,4)*{};(.5,6)*{}**\dir{-};
	(-1.5,4)*{};(-.5,6)*{}**\dir{-};
	(0,0)*{};(0,2)*{}**\dir{-};
	(1.5,1.5)*{\scriptstyle 2};
	(-3,3.5)*{\scriptstyle 1};
	(3,3.5)*{\scriptstyle 3};
	\endxy}\otimes |
+(\raisebox{-.1mm}{\xy 0;/r.3pc/:
	(1.5,4)*{};(0,2)*{}**\dir{-};
	(-1.5,4)*{};(0,2)*{}**\dir{-};
	(0,0)*{};(0,2)*{}**\dir{-};(1.5,1.5)*{\scriptstyle 1};
	\endxy}*\raisebox{-.1mm}{\xy 0;/r.3pc/:
	(1.5,4)*{};(0,2)*{}**\dir{-};
	(-1.5,4)*{};(0,2)*{}**\dir{-};
	(0,0)*{};(0,2)*{}**\dir{-};(1.5,1.5)*{\scriptstyle 3};
	\endxy})\otimes \raisebox{-.1mm}{\xy 0;/r.3pc/:
	(1.5,4)*{};(0,2)*{}**\dir{-};
	(-1.5,4)*{};(0,2)*{}**\dir{-};
	(0,0)*{};(0,2)*{}**\dir{-};(1.5,1.5)*{\scriptstyle 2};
	\endxy}
+\raisebox{-.1mm}{\xy 0;/r.3pc/:
	(1.5,4)*{};(0,2)*{}**\dir{-};
	(-1.5,4)*{};(0,2)*{}**\dir{-};
	(0,0)*{};(0,2)*{}**\dir{-};(1.5,1.5)*{\scriptstyle 1};
	\endxy}\otimes \raisebox{-.1mm}{\xy 0;/r.3pc/:
	(3,6)*{};(0,2)*{}**\dir{-};
	(-3,6)*{};(0,2)*{}**\dir{-};
	(1.5,4)*{};(0,6)*{}**\dir{-};
	(0,0)*{};(0,2)*{}**\dir{-};
	(1.5,1.5)*{\scriptstyle 2};
	(3,3.5)*{\scriptstyle 3};
	\endxy}
+\raisebox{-.1mm}{\xy 0;/r.3pc/:
	(1.5,4)*{};(0,2)*{}**\dir{-};
	(-1.5,4)*{};(0,2)*{}**\dir{-};
	(0,0)*{};(0,2)*{}**\dir{-};(1.5,1.5)*{\scriptstyle 3};
	\endxy}\otimes\raisebox{-.1mm}{\xy 0;/r.3pc/:
	(3,6)*{};(0,2)*{}**\dir{-};
	(-3,6)*{};(0,2)*{}**\dir{-};
	(-1.5,4)*{};(0,6)*{}**\dir{-};
	(0,0)*{};(0,2)*{}**\dir{-};
	(1.5,1.5)*{\scriptstyle 2};
	(-3,3.5)*{\scriptstyle 1};
	\endxy}
+|\otimes\raisebox{-.1mm}{\xy 0;/r.3pc/:
	(3,6)*{};(0,2)*{}**\dir{-};
	(-3,6)*{};(0,2)*{}**\dir{-};
	(1.5,4)*{};(.5,6)*{}**\dir{-};
	(-1.5,4)*{};(-.5,6)*{}**\dir{-};
	(0,0)*{};(0,2)*{}**\dir{-};
	(1.5,1.5)*{\scriptstyle 2};
	(-3,3.5)*{\scriptstyle 1};
	(3,3.5)*{\scriptstyle 3};
	\endxy}\\
&=\raisebox{-.1mm}{\xy 0;/r.3pc/:
	(3,6)*{};(0,2)*{}**\dir{-};
	(-3,6)*{};(0,2)*{}**\dir{-};
	(1.5,4)*{};(.5,6)*{}**\dir{-};
	(-1.5,4)*{};(-.5,6)*{}**\dir{-};
	(0,0)*{};(0,2)*{}**\dir{-};
	(1.5,1.5)*{\scriptstyle 2};
	(-3,3.5)*{\scriptstyle 1};
	(3,3.5)*{\scriptstyle 3};
	\endxy}\otimes |
+(\raisebox{-.1mm}{\xy 0;/r.3pc/:
	(3,6)*{};(0,2)*{}**\dir{-};
	(-3,6)*{};(0,2)*{}**\dir{-};
	(-1.5,4)*{};(0,6)*{}**\dir{-};
	(0,0)*{};(0,2)*{}**\dir{-};
	(1.5,1.5)*{\scriptstyle 3};
	(-3,3.5)*{\scriptstyle 1};
	\endxy}+\raisebox{-.1mm}{\xy 0;/r.3pc/:
	(3,6)*{};(0,2)*{}**\dir{-};
	(-3,6)*{};(0,2)*{}**\dir{-};
	(1.5,4)*{};(0,6)*{}**\dir{-};
	(0,0)*{};(0,2)*{}**\dir{-};
	(1.5,1.5)*{\scriptstyle 1};
	(3,3.5)*{\scriptstyle 3};
	\endxy})\otimes \raisebox{-.1mm}{\xy 0;/r.3pc/:
	(1.5,4)*{};(0,2)*{}**\dir{-};
	(-1.5,4)*{};(0,2)*{}**\dir{-};
	(0,0)*{};(0,2)*{}**\dir{-};(1.5,1.5)*{\scriptstyle 2};
	\endxy}
+\raisebox{-.1mm}{\xy 0;/r.3pc/:
	(1.5,4)*{};(0,2)*{}**\dir{-};
	(-1.5,4)*{};(0,2)*{}**\dir{-};
	(0,0)*{};(0,2)*{}**\dir{-};(1.5,1.5)*{\scriptstyle 1};
	\endxy}\otimes \raisebox{-.1mm}{\xy 0;/r.3pc/:
	(3,6)*{};(0,2)*{}**\dir{-};
	(-3,6)*{};(0,2)*{}**\dir{-};
	(1.5,4)*{};(0,6)*{}**\dir{-};
	(0,0)*{};(0,2)*{}**\dir{-};
	(1.5,1.5)*{\scriptstyle 2};
	(3,3.5)*{\scriptstyle 3};
	\endxy}
+\raisebox{-.1mm}{\xy 0;/r.3pc/:
	(1.5,4)*{};(0,2)*{}**\dir{-};
	(-1.5,4)*{};(0,2)*{}**\dir{-};
	(0,0)*{};(0,2)*{}**\dir{-};(1.5,1.5)*{\scriptstyle 3};
	\endxy}\otimes\raisebox{-.1mm}{\xy 0;/r.3pc/:
	(3,6)*{};(0,2)*{}**\dir{-};
	(-3,6)*{};(0,2)*{}**\dir{-};
	(-1.5,4)*{};(0,6)*{}**\dir{-};
	(0,0)*{};(0,2)*{}**\dir{-};
	(1.5,1.5)*{\scriptstyle 2};
	(-3,3.5)*{\scriptstyle 1};
	\endxy}
+|\otimes\raisebox{-.1mm}{\xy 0;/r.3pc/:
	(3,6)*{};(0,2)*{}**\dir{-};
	(-3,6)*{};(0,2)*{}**\dir{-};
	(1.5,4)*{};(.5,6)*{}**\dir{-};
	(-1.5,4)*{};(-.5,6)*{}**\dir{-};
	(0,0)*{};(0,2)*{}**\dir{-};
	(1.5,1.5)*{\scriptstyle 2};
	(-3,3.5)*{\scriptstyle 1};
	(3,3.5)*{\scriptstyle 3};
	\endxy}.
\end{split}\]

Now for any $Y(v_1,\dots,v_n)\in\bar{B}(V)$ and $P\in\calf_Y$, let $d_P$ be the common vertex degree of the terms in $P^*$. By Eq.~\meqref{lrco1} and the canonical order~\meqref{eq:yset} there is clearly a unique $w_P\in\frakS_{d_P,n-d_P}$
such that
\[\Delta_{BT}(Y)=\sum_{P\in\calf_Y}P^*(v_{w_P(1)},\dots,v_{w_P(d_P)})\otimes(Y/P)(v_{w_P(d_P+1)},\dots,v_{w_P(n)}).\]
We then obtain the formula
\begin{equation}\mlabel{blrco}
\Delta'_{BT}(Y)=\sum_{P\in\calf_Y}(P^*\otimes(Y/P))T^\sigma_{w_P^{-1}}(v_1\otimes\cdots\otimes v_n),
\end{equation}
for any $Y(v_1,\dots,v_n)\in\bar{\calb}_n(V),v_1,\dots,v_n\in V,\,n\geq1$. To finish this section, we show that $(\caly(V),*_\sigma)$ is a free associative algebra as an analogue of \cite[Theorem 3.8]{LR1}.
\begin{theorem}\mlabel{fabt}
Given a braided vector space $(V,\sigma)$ with a $\bk$-basis $E=\{v_x\}_{x\in X}$, the algebra $(\caly(V),*_\sigma)$ is freely generated by the set $\calw(E):=\sqcup_{n\geq1}\calw_n\times E^n$, where $$\calw_n:=\{\,|\vee Y'\in\calb_n\,|\,Y'\in\calb_{n-1}\}, \quad n\geq1.$$
\end{theorem}
\begin{proof}
Note that $\bk\calw(E)$ is a braided vector subspace of $\caly(V)$. Thus there is a braided algebra map
\[f:T(\bk\calw(E))\rar\caly(V),\,Y\mapsto Y,\,Y\in\calw(E),\]
by the freeness of the tensor algebra. First we show that $f$ is surjective. Let $Y=Y_1\vee_v Y_2$. If $Y_1=|\,$, then $Y\in\bk\calw(E)$. Otherwise, let $Y_1=Y_{11}\vee_u Y_{12}$ with $Y_{11}, Y_{12}\in \calw(E)$. Then
\begin{equation}\mlabel{eq:gen}
Y=Y_1\vee_v Y_2=
(Y_1\succ_\sigma(\,|\vee_v Y_2))\\
=Y_1*_\sigma(\,|\vee_v Y_2)-Y_{11}\vee_u(Y_{12}*_\sigma(\,|\vee_v Y_2)).
\end{equation}
Since $Y_{11}$ has degree less then that of $Y_1$, by the induction hypothesis on $d_v(Y_1)$, we see that $Y_1$ and $Y_{11}\vee_u(Y_{12}*_\sigma(\,|\vee Y_2))$ are in the subalgebra of $\caly(V)$ generated by $\calw(E)$. Thus $Y$ is also in this subalgebra. Since $Y$ is arbitrary, we see that $f$ is surjective.

On the other hand, it has been shown in~\cite[Theorem 3.8]{LR1} that $\dim T(\bk\calw)_n=C_n=|\calb_n|$\,, while we also have $T(\bk\calw(E))_n\cong T(\bk\calw)_n\otimes V^{\otimes n}$. Thus $f$ is also injective, mapping a linear basis of $T(\bk\calw(E))$ to one in $\caly(V)$.
\end{proof}	

\section{Comparison isomorphism between braided Hopf algebras of rooted trees}
\mlabel{sec:comp}

Holtkamp~\mcite{Hol} gave a canonical isomorphism between the Foissy-Holtkamp Hopf algebra of planar rooted trees, as the noncommutative variation of the Connes-Kreimer Hopf algebra of (nonplanar) rooted trees, and the Loday-Ronco Hopf algebra of planar binary rooted trees. With the braided version of the Loday-Ronco Hopf algebra constructed in the previous section, it is natural to find a braided Hopf algebra of planar rooted trees that is isomorphic to the braided Hopf algebra of planar binary trees. In the work~\mcite{Fo3,Fo4}, Foissy gave a braided Hopf algebra of planar rooted trees based on certain linear ordering of the vertices of the rooted trees. Here we provide another braiding based on a different linear ordering for which the desired isomorphism of braided Hopf algebras can be established (Theorem~\mref{thm:iso}).

\subsection{Braided operated Hopf algebras of planar rooted trees}
\mlabel{ss:planar}

Let $\calt(X)$ be the set of planar rooted trees with vertices decorated by elements in $X$ and $\calf(X)=M(\calt(X))$ be the free monoid generated by $\calt(X)$ representing the set of planar rooted forests as concatenations of rooted trees.
Similar to (but different from) the binary case, define the {\bf (vertex) degree} $d_v(F)$ of $F\in \calf(X)$ to be the number of its vertices. For $n\geq 0$, let $\calf_n(X)$ be the subset of $\calf(X)$ consisting of forests with $n$ vertices. Similarly, we denote $\calt_n(X)$ for trees with $n$ vertices. In particular, $\calt_0(X)=\calf_0(X)=\{\emptyset\}$ with $\emptyset$ being the empty tree, and $\calt_1(X)=\calf_1(X)=\{\,\text{\raisebox{2pt}{$\sbu_x$}}\,|\,x\in X\}$.

When $X$ is a singleton, we abbreviate $\calt(X)$ (resp. $\calf(X)$) as $\calt$ (resp. $\calf$) consisting of planar rooted trees (resp. forests) with the same, and therefore free of, decoration. Then $|\calt_n|$ is the $(n-1)$-th Catalan number $C_{n-1}=\tfrac{1}{n}{2(n-1)\choose n-1},\,n\geq1$~\cite[Exercise 6.19. e.]{Sta}. Further, $|\calf_n|=|\calt_{n+1}|=C_n,\,n\geq0$, since there is a one-to-one correspondence between $\calf_n$ and $\calt_{n+1}$ sending $F\in \calf_n$ to its grafted tree $B^+(F)\in\calt_{n+1}$. Here $B^+$ is the grafting operator.

For the braiding consideration next, we need to fix a well-order on the set of vertices of every forest $F\in\calf$. One of the choices is given by doing a depth-first (preorder) search through trees in $F$ from left to right and numbering the occurring vertices successively. We called this the {\bf canonical order} for planar rooted forests. Then we can denote
$F(x_1,\dots,x_n)\in\calf_n(X)$
for the forest $F$ with its canonically ordered vertices decorated by $x_1,\dots,x_n\in X$ 
, and have the identification
\begin{equation}
\calf_n (X)= \calf_n\times X^n,\, F(x_1,\cdots,x_n) \leftrightarrow (F,(x_1,\cdots,x_n)), \quad F\in \calf_n, x_1,\cdots, x_n\in X.
\mlabel{eq:ftree}
\end{equation}
If there is no danger of confusion, we also use $F$ to denote $F(x_1,\dots,x_n)$ for brevity.

For example, the preorder of one $F\in \calf_7$ and the corresponding decorated forest are as follows.
\[F=\raisebox{-.1mm}{\xy 0;/r.3pc/:
	(-3,0)*{};(-3,2)*{}**\dir{-};
	(1.5,2)*{};(0,4)*{}**\dir{-};
	(1.5,2)*{};(3,4)*{}**\dir{-};
	(1.5,2)*{};(0,0)*{}**\dir{-};
	(-1.5,2)*{};(0,0)*{}**\dir{-};
	 (-1.5,2)*{\sbu};(1.5,2)*{\sbu};(0,4)*{\sbu};(3,4)*{\sbu};(0,0)*{\sbu};(-3,0)*{\sbu};(-3,2)*{\sbu};
	(-4.5,-.5)*{\scriptstyle 1};(-4.5,1.5)*{\scriptstyle 2};(1.5,-.5)*{\scriptstyle 3};(-1.5,3.7)*{\scriptstyle 4};(3,1.5)*{\scriptstyle 5};(0,5.7)*{\scriptstyle 6};(3,5.7)*{\scriptstyle 7};
	\endxy},\qquad
F(x_1,\dots,x_7)=\raisebox{-.1mm}{\xy 0;/r.3pc/:
	(-3,0)*{};(-3,2)*{}**\dir{-};
	(1.5,2)*{};(0,4)*{}**\dir{-};
	(1.5,2)*{};(3,4)*{}**\dir{-};
	(1.5,2)*{};(0,0)*{}**\dir{-};
	(-1.5,2)*{};(0,0)*{}**\dir{-};
	 (-1.5,2)*{\sbu};(1.5,2)*{\sbu};(0,4)*{\sbu};(3,4)*{\sbu};(0,0)*{\sbu};(-3,0)*{\sbu};(-3,2)*{\sbu};
	(-4.7,-.5)*{\scriptstyle x_1};(-4.7,1.5)*{\scriptstyle x_2};(1.7,-.5)*{\scriptstyle x_3};(-1.5,3.7)*{\scriptstyle x_4};(3.2,1.5)*{\scriptstyle x_5};(0,5.7)*{\scriptstyle x_6};(3,5.7)*{\scriptstyle x_7};
	\endxy}\]
\indent For a planar rooted tree $T$, a {\bf subforest} $F$ of $T$ is the forest consisting of a subset of vertices of $T$ together with their descents and edges containing all these involved vertices. Let $\calf_T$ be the set of subforests of $T$, including the empty tree $\emptyset$ and the total $T$.

\smallskip
A Hopf algebra structure on planar rooted trees was established by Foissy and Holtkamp~\mcite{Fo4,Hol} as a noncommutative variation of the Hopf algebra on (nonplanar) rooted trees from the Connes-Kreimer approach of the renormalization of perturbative quantum field theory~\mcite{CK,Kre}.
\begin{defn}
The {\bf noncommutative Connes-Kreimer Hopf algebra} (a.k.a {\bf Foissy-Holtkamp Hopf algebra})  $\calh_{RT}:=\calh_{RT}(X):=\bk\calf(X)$ is the free $\bk$-algebra generated by the set $\calt(X)$ of planar rooted trees decorated by $X$, with the coproduct $\Delta_{RT}$ defined by
\begin{equation}\mlabel{ckco}
\Delta_{RT}(T)=\sum_{F\in\calf_T}F\otimes(T/F),\,T\in\calt(X),
\end{equation}
and counit $\varepsilon_{RT}(T)=\delta_{T,\emptyset}$, where $T/F$ is the tree obtained from $T$ after removing a subforest $F$ and all edges connecting $F$ to the rest of $T$. Here the decorations on vertices of $T$ remain at their corresponding positions in $F$ and $T/F$.
\mlabel{nck}
\end{defn}

For example, if $T=\raisebox{-.1mm}{\xy 0;/r.3pc/:
	(1.5,2)*{};(0,4)*{}**\dir{-};
	(1.5,2)*{};(3,4)*{}**\dir{-};
	(1.5,2)*{};(0,0)*{}**\dir{-};
	(-1.5,2)*{};(0,0)*{}**\dir{-};
	 (-1.5,2)*{\sbu};(1.5,2)*{\sbu};(0,4)*{\sbu};(3,4)*{\sbu};(0,0)*{\sbu};
	(1.7,-.5)*{\scriptstyle x_1};(-1.5,3.7)*{\scriptstyle x_2};(3.2,1.5)*{\scriptstyle x_3};(0,5.7)*{\scriptstyle x_4};(3,5.7)*{\scriptstyle x_5};
	\endxy}$, then
\begin{equation}\mlabel{ex:co}
\begin{split}
\Delta_{RT}(T)&=\raisebox{-.1mm}{\xy 0;/r.3pc/:
	(1.5,2)*{};(0,4)*{}**\dir{-};
	(1.5,2)*{};(3,4)*{}**\dir{-};
	(1.5,2)*{};(0,0)*{}**\dir{-};
	(-1.5,2)*{};(0,0)*{}**\dir{-};
	 (-1.5,2)*{\sbu};(1.5,2)*{\sbu};(0,4)*{\sbu};(3,4)*{\sbu};(0,0)*{\sbu};
	(1.7,-.5)*{\scriptstyle x_1};(-1.5,3.7)*{\scriptstyle x_2};(3.2,1.5)*{\scriptstyle x_3};(0,5.7)*{\scriptstyle x_4};(3,5.7)*{\scriptstyle x_5};
	\endxy}\otimes 1
+\raisebox{-.1mm}{\xy 0;/r.3pc/:
	(1.5,2)*{};(0,0)*{}**\dir{-};
	(-1.5,2)*{};(0,0)*{}**\dir{-};
	(-1.5,2)*{\sbu};(1.5,2)*{\sbu};(0,0)*{\sbu};(-3,0)*{\sbu};
	(-3,1.7)*{\scriptstyle x_2};(1.7,-.5)*{\scriptstyle x_3};(-1.5,3.7)*{\scriptstyle x_4};(1.5,3.7)*{\scriptstyle x_5};
	\endxy}\otimes\sbu_{x_1}
+\raisebox{-.1mm}{\xy 0;/r.3pc/:
	(1.5,2)*{};(0,0)*{}**\dir{-};
	(-1.5,2)*{};(0,0)*{}**\dir{-};
	(-1.5,2)*{\sbu};(1.5,2)*{\sbu};(0,0)*{\sbu};
	(1.7,-.5)*{\scriptstyle x_3};(-1.5,3.7)*{\scriptstyle x_4};(1.5,3.7)*{\scriptstyle x_5};
	\endxy}\otimes
\raisebox{-.1mm}{\xy 0;/r.3pc/:
	(0,2)*{};(0,0)*{}**\dir{-};
	(0,2)*{\sbu};(0,0)*{\sbu};
	(1.5,-.5)*{\scriptstyle x_1};(1.5,1.5)*{\scriptstyle x_2};
	\endxy}
+\sbu_{x_2}\sbu_{x_4}\sbu_{x_5}\otimes\raisebox{-.1mm}{\xy 0;/r.3pc/:
	(0,2)*{};(0,0)*{}**\dir{-};
	(0,2)*{\sbu};(0,0)*{\sbu};
	(1.5,-.5)*{\scriptstyle x_1};(1.5,1.5)*{\scriptstyle x_3};
	\endxy}
+\sbu_{x_4}\sbu_{x_5}\otimes\raisebox{-.1mm}{\xy 0;/r.3pc/:
	(1.5,2)*{};(0,0)*{}**\dir{-};
	(-1.5,2)*{};(0,0)*{}**\dir{-};
	(-1.5,2)*{\sbu};(1.5,2)*{\sbu};(0,0)*{\sbu};
	(1.7,-.5)*{\scriptstyle x_1};(-1.5,3.7)*{\scriptstyle x_2};(1.5,3.7)*{\scriptstyle x_3};
	\endxy}\\
&\quad+\sbu_{x_2}\sbu_{x_4}\otimes\raisebox{-.1mm}{\xy 0;/r.3pc/:
	(0,2)*{};(0,0)*{}**\dir{-};
	(0,4)*{};(0,2)*{}**\dir{-};
	(0,4)*{\sbu};(0,2)*{\sbu};(0,0)*{\sbu};
	(1.5,-.5)*{\scriptstyle x_1};(1.5,1.5)*{\scriptstyle x_3};(1.5,3.5)*{\scriptstyle x_5};
	\endxy}
+\sbu_{x_2}\sbu_{x_5}\otimes \raisebox{-.1mm}{\xy 0;/r.3pc/:
	(0,2)*{};(0,0)*{}**\dir{-};
	(0,4)*{};(0,2)*{}**\dir{-};
	(0,4)*{\sbu};(0,2)*{\sbu};(0,0)*{\sbu};
	(1.5,-.5)*{\scriptstyle x_1};(1.5,1.5)*{\scriptstyle x_3};(1.5,3.5)*{\scriptstyle x_4};
	\endxy}
+\sbu_{x_2}\otimes\raisebox{-.1mm}{\xy 0;/r.3pc/:
	(0,2)*{};(-1.5,4)*{}**\dir{-};
	(0,2)*{};(1.5,4)*{}**\dir{-};
	(0,2)*{};(0,0)*{}**\dir{-};
	(0,2)*{\sbu};(-1.5,4)*{\sbu};(1.5,4)*{\sbu};(0,0)*{\sbu};
	(1.7,-.5)*{\scriptstyle x_1};(1.7,1.5)*{\scriptstyle x_3};(-1.5,5.7)*{\scriptstyle x_4};(1.5,5.7)*{\scriptstyle x_5};
	\endxy}
+\sbu_{x_4}\otimes\raisebox{-.1mm}{\xy 0;/r.3pc/:
	(1.5,2)*{};(1.5,4)*{}**\dir{-};
	(1.5,2)*{};(0,0)*{}**\dir{-};
	(-1.5,2)*{};(0,0)*{}**\dir{-};
	(-1.5,2)*{\sbu};(1.5,2)*{\sbu};(1.5,4)*{\sbu};(0,0)*{\sbu};
	(1.7,-.5)*{\scriptstyle x_1};(-1.5,3.7)*{\scriptstyle x_2};
	(3,1.5)*{\scriptstyle x_3};(3,3.5)*{\scriptstyle x_5};
	\endxy}
+\sbu_{x_5}\otimes\raisebox{-.1mm}{\xy 0;/r.3pc/:
	(1.5,2)*{};(1.5,4)*{}**\dir{-};
	(1.5,2)*{};(0,0)*{}**\dir{-};
	(-1.5,2)*{};(0,0)*{}**\dir{-};
	(-1.5,2)*{\sbu};(1.5,2)*{\sbu};(1.5,4)*{\sbu};(0,0)*{\sbu};
	(1.7,-.5)*{\scriptstyle x_1};(-1.5,3.7)*{\scriptstyle x_2};
	(3,1.5)*{\scriptstyle x_3};(3,3.5)*{\scriptstyle x_4};
	\endxy}
+1\otimes \raisebox{-.1mm}{\xy 0;/r.3pc/:
	(1.5,2)*{};(0,4)*{}**\dir{-};
	(1.5,2)*{};(3,4)*{}**\dir{-};
	(1.5,2)*{};(0,0)*{}**\dir{-};
	(-1.5,2)*{};(0,0)*{}**\dir{-};
	 (-1.5,2)*{\sbu};(1.5,2)*{\sbu};(0,4)*{\sbu};(3,4)*{\sbu};(0,0)*{\sbu};
	(1.7,-.5)*{\scriptstyle x_1};(-1.5,3.7)*{\scriptstyle x_2};(3.2,1.5)*{\scriptstyle x_3};(0,5.7)*{\scriptstyle x_4};(3,5.7)*{\scriptstyle x_5};
	\endxy}.
\end{split}
\end{equation}

As in the non-ordered case, it is easy to check that
$\calh_{RT}:=\bigoplus_{n\geq 0}\calh_{RT,\,n}:=\bigoplus_{n\geq 0}\bk\calf_n(X)$
is a graded connected Hopf algebra with respect to the grading defined by the degree $d_v$. In particular, $\calh_{RT,\,0}=\bk \emptyset$ is spanned by the empty tree $\emptyset$, and $\dim\calh_{RT,\,n}=C_n|X|^n,\,n\geq1$, when $|X|<\infty$.

Given $x\in X$, one can also define a grafting operator $B^+_x$ on
$\calh_{RT}$ by grafting a decorated forest $F\in \calf(X)$ to a decorated rooted tree $B^+_x(F)$ with the root decorated by $x$.
Then the grafting operator satisfies the following {\bf 1-cocycle} condition
\begin{equation}\mlabel{cocy}
\Delta_{RT}B^+_x=B^+_x\otimes1+(\Id_{\calh_{RT}}\otimes B^+_x)\Delta_{RT}.
\end{equation}

For the braided Hopf algebra of planar rooted trees, we first introduce some notations, referring to the construction of Foissy in \mcite{Fo3,Fo4}.
Like the notation $\caly(V)$ for a vector space $V$, define
\[\calr(V):=\bigoplus_{n\geq0}\calr_n(V):=\bigoplus_{n\geq0}\bk\calf_n\otimes V^{\otimes n},\]
with $\calr_0(V)=\bk$ by convention.

Also, we can identify $F\ot(v_1\ot\cdots\ot v_n)$ with $F(v_1,\dots,v_n)$ via the canonical order, and denote  \[\bar{\calf}_n(V):=\{F(v_1,\dots,v_n)\,|\,F\in\calf_n,v_1,\dots,v_n\in V\}\]
for any $n\geq1$ and $\bar{\calf}(V):=\sqcup_{n\geq0}\bar{\calf}_n(V)$ with $\bar{\calf}_0(V)=\{\emptyset\}$.
In particular, let $\bar{\calt}_n(V)$ be the subset of $\bar{\calf}_n(V)$ only consisting of trees $T(v_1,\dots,v_n)$, and $\bar{\calt}(V):=\sqcup_{n\geq0}\bar{\calt}_n(V)$ with $\bar{\calt}_0(V)=\{\emptyset\}$.

Note that the concatenation product $\cdot$, the coproduct $\Delta_{RT}$ and the counit $\varepsilon_{RT}$ of $\calh_{RT}$ can be linearly extended on $\calr(V)$ to obtain the induced Hopf algebra structure $(\calr(V),\cdot,\Delta_{RT},\varepsilon_{RT})$.

\medskip
Now if $(V,\sigma)$ is a braided vector space,
by \cite[Proposition 17]{Fo3} the braiding $\sigma$ on $V$ induces a braiding $\sigma_{RT}$ on $\calr(V)$ defined by
\[\sigma_{RT}(F(v_1,\dots,v_m)\otimes G(v_{m+1},\dots,v_{m+n}))=(G\otimes F)\beta_{mn}(v_1\otimes\cdots\otimes v_{m+n})\]
for any $F\in\calf_m,G\in\calf_n$ and $v_1,\dots,v_{m+n}\in V$, where we similarly interpret any $F\in \calf_m$ as a linear function from $V^{\ot m}$ to $\calr(V)$ with
\[F\left(\sum c_{i_1,\dots,i_m}v_{i_1}\ot \cdots \ot v_{i_m}\right)=\sum c_{i_1,\dots,i_m}F(v_{i_1},\dots,v_{i_m}),\,c_{i_1,\dots,i_m}\in\bk.\]
Also, we have the linear map
\[\sigma_{V,\calr(V)}:V\otimes\calr(V)\rar \calr(V)\otimes V,\,v\otimes F(v_1,\dots,v_m)\mapsto (F\otimes\Id_V)\beta_{1,\,m}(v\otimes v_1\otimes\cdots\otimes v_m)\]
for any $F\in\calf_m$ and $v,v_1,\dots,v_m\in V$. By symmetry, we define $\sigma_{\calr(V),V}:\calr(V)\otimes V\rar V\otimes\calr(V)$.

On the other hand, according to Eq.~\meqref{ckco} and the canonical order in Eq.~\meqref{eq:ftree}, for any $T(v_1,\dots,v_n)\in\bar{\calt}_n(V)$, every $F\in\calf_T$ with $d_F:=d_v(F)$ determines a unique $w_F\in\frakS_{d_F,n-d_F}$ such that
\[\Delta_{RT}(T)=\sum_{F\in\calf_T}F(v_{w_F(1)},\dots,v_{w_F(d_F)})\otimes(T/F)(v_{w_F(d_F+1)},\dots,v_{w_F(n)});\]
see also the example in Eq.~\meqref{ex:co} as an illustration. Hence we can define a linear map $\Delta'_{RT}:\calr(V)\rar\calr(V)\otimes\calr(V)$ as follows.
For any $T(v_1,\dots,v_n)\in\bar{\calt}_n(V),\,n\geq1$, define
\begin{equation}\mlabel{bckco}
\Delta'_{RT}(T)=\sum_{F\in\calf_T}(F\otimes(T/F))T^\sigma_{w_F^{-1}}(v_1\otimes\cdots\otimes v_n);
\end{equation}
while for any $F\in\bar{\calf}(V)$, we define it by induction on the number of trees in $F$,  namely,
\[\Delta'_{RT}(F)=(\Delta'_{RT}\otimes\Delta'_{RT})\sigma_{RT}(T\otimes F')\]
when $F=T\cdot F'$ is the concatenation of  $T\in\bar{\calt}(V)$ and $F'\in\bar{\calf}(V)$, and then extend it to $\calr(V)$ linearly. Also, define a linear map $\varepsilon'_{RT}:\calr(V)\rar\bk$ by $\varepsilon'_{RT}(F)=\delta_{F,\emptyset}$. As a result, the quintuple $(\calr(V),\cdot\,,\Delta'_{RT},\varepsilon'_{RT},\sigma_{RT})$ is a braided Hopf algebra. See~\cite[Theorem 19]{Fo3}.

\medskip

Let $B^+:V\otimes\calr(V)\rar \calr(V)$ be the linear map defined by $B^+(v\otimes F):=B^+_v(F)$ for any $v\in V, F\in\bar{\calf}(V)$.
For simplicity, we denote $B^+(v\otimes F)$ by $B^+_v(F)$. So $(\calr(V),B^+)$ is an $V$-algebra over $V$ in the sense of~\mcite{Fo4} or a $V$-operated algebra in the sense of~\mcite{Gu2}. Moreover, $B^+$ satisfies the following {\bf twisted 1-cocycle} condition analogous to Eq.~\meqref{cocy},
\begin{equation}\mlabel{bcoc}
\Delta'_{RT}B^+=B^+\otimes 1+(\Id_{\calr(V)}\otimes B^+)(\sigma_{V,\calr(V)})_1(\Id_V\otimes \Delta'_{RT}),
\end{equation}
and the property below.

\begin{lemma}
For any $v\in V,\,F,F'\in\bar{\calf}(V)$, we have
\begin{align}
&\mlabel{fv1}\sigma_{RT}(B^+_v(F)\otimes F')=(\Id_{\calr(V)}\otimes B^+)(\sigma_{V,\calr(V)})_1(\sigma_{RT})_2(v\otimes F\otimes F'),\\
&\mlabel{fv2}\sigma_{RT}(F\otimes B^+_v(F'))=(B^+\otimes\Id_{\calr(V)})(\sigma_{RT})_2(\sigma_{\calr(V),V})_1(F\otimes v\otimes F').
\end{align}		
\end{lemma}
\begin{proof}
It is by the definitions of the grafting operator $B^+$ and braidings $\sigma_{RT},\,\sigma_{V,\calr(V)},\,\sigma_{\calr(V),V}$.
\end{proof}

\begin{remark}
The braided  Hopf algebra $\calr(V)$ of planar rooted trees considered here is slightly different from the one defined in \mcite{Fo3,Fo4} denoted $\calh(V)$, as our canonical order on the set of vertices of a forest is not the one in~\cite[Sec. 2.2.2]{Fo4}. Indeed, our choice makes the resulting braided Hopf algebra isomorphic to the braided Hopf algebra $\caly(V)$ of planar binary trees shown below. See also Lemma \mref{ord} on the ordering.	
\mlabel{rk:foi}
\end{remark}

\subsection{Comparison of braided Hopf algebras on trees}
In \cite[Theorem 2.10]{Hol}, Holtkamp proved that there is a graded Hopf algebra isomorphism $\Theta:\calh_{BT}\rar\calh_{RT}$ defined recursively by
\begin{equation}\mlabel{theta}
\Theta(\,|\,)=1,\,\Theta(\,|\vee Y)=(B^+\circ\Theta)(Y),\,Y\in\calb.
\end{equation}
First we note that Holtkamp's map $\Theta$ preserves the canonical orders given in Eqs.~\meqref{eq:yset} and \meqref{eq:ftree} for vertices of trees in $\calh_{BT}$ and $\calh_{RT}$ respectively. More precisely,
\begin{lemma}
Let $Y\in\calb_n$ with its $n$ internal vertices canonically ordered.  Then $\Theta(Y)\in\calh_{RT}=\bk\calf$ is a sum of monomial terms as forests in which the $n$ vertices are also canonically ordered.
\mlabel{ord}
\end{lemma}
\begin{proof}
Indeed, we can see this fact by induction on $d_v$. Let $X$ be a set and $Y\in\calb_n(X)$ with its canonically ordered internal vertices decorated by $x_2,\dots,x_{n+1}\in X$, then
\[\Theta(\,|\vee_{x_1} Y)=(B^+_{x_1}\circ\Theta)(Y)=B^+_{x_1}\left(\sum_{F\in\calf_n}c_F F(x_2,\dots,x_{n+1})\right)=
\sum_{F\in\calf_n}c_FB^+(F)(x_1,\dots,x_{n+1}),\]
with $\Theta(Y)=\sum_{F\in\calf_n}c_F F(x_2,\dots,x_{n+1}),\,c_F\in\bk$, by the induction hypothesis. Namely, all monomial terms in $\Theta(\,|\vee_{x_1} Y)$ still have their vertices decorated by $x_1,\dots,x_{n+1}$ in the canonical order.
\end{proof}

For example, by Eq.~\meqref{eq:gen} in the proof of Theorem \mref{fabt} and the definition \meqref{theta} of $\Theta$,
\[\Theta(\raisebox{-.6mm}{\xy 0;/r.3pc/:
	(4.5,8)*{};(0,2)*{}**\dir{-};
	(-4.5,8)*{};(0,2)*{}**\dir{-};
	(0,6)*{};(-1.5,8)*{}**\dir{-};
	(1.5,8)*{};(-1.5,4)*{}**\dir{-};
	(0,0)*{};(0,2)*{}**\dir{-};
	(-3,3.5)*{\scriptstyle 1};
	(1.5,6)*{\scriptstyle 2};
	(1.5,1.5)*{\scriptstyle 3};
	\endxy})
=\Theta(\raisebox{-.1mm}{\xy 0;/r.3pc/:
	(3,6)*{};(0,2)*{}**\dir{-};
	(-3,6)*{};(0,2)*{}**\dir{-};
	(1.5,4)*{};(0,6)*{}**\dir{-};
	(0,0)*{};(0,2)*{}**\dir{-};
	(1.5,1.5)*{\scriptstyle 1};
	(3,3.5)*{\scriptstyle 2};
	\endxy}*\raisebox{-.1mm}{\xy 0;/r.3pc/:
	(1.5,4)*{};(0,2)*{}**\dir{-};
	(-1.5,4)*{};(0,2)*{}**\dir{-};
	(0,0)*{};(0,2)*{}**\dir{-};(1.5,1.5)*{\scriptstyle 3};
	\endxy}-\,|\vee_1(\raisebox{-.1mm}{\xy 0;/r.3pc/:
	(1.5,4)*{};(0,2)*{}**\dir{-};
	(-1.5,4)*{};(0,2)*{}**\dir{-};
	(0,0)*{};(0,2)*{}**\dir{-};(1.5,1.5)*{\scriptstyle 2};
	\endxy}*\raisebox{-.1mm}{\xy 0;/r.3pc/:
	(1.5,4)*{};(0,2)*{}**\dir{-};
	(-1.5,4)*{};(0,2)*{}**\dir{-};
	(0,0)*{};(0,2)*{}**\dir{-};(1.5,1.5)*{\scriptstyle 3};
	\endxy}))
=B^+_1(\Theta(\raisebox{-.1mm}{\xy 0;/r.3pc/:
	(1.5,4)*{};(0,2)*{}**\dir{-};
	(-1.5,4)*{};(0,2)*{}**\dir{-};
	(0,0)*{};(0,2)*{}**\dir{-};(1.5,1.5)*{\scriptstyle 2};
	\endxy}))\Theta(\raisebox{-.1mm}{\xy 0;/r.3pc/:
	(1.5,4)*{};(0,2)*{}**\dir{-};
	(-1.5,4)*{};(0,2)*{}**\dir{-};
	(0,0)*{};(0,2)*{}**\dir{-};(1.5,1.5)*{\scriptstyle 3};
	\endxy})-B^+_1(\Theta(\raisebox{-.1mm}{\xy 0;/r.3pc/:
	(1.5,4)*{};(0,2)*{}**\dir{-};
	(-1.5,4)*{};(0,2)*{}**\dir{-};
	(0,0)*{};(0,2)*{}**\dir{-};(1.5,1.5)*{\scriptstyle 2};
	\endxy})\Theta(\raisebox{-.1mm}{\xy 0;/r.3pc/:
	(1.5,4)*{};(0,2)*{}**\dir{-};
	(-1.5,4)*{};(0,2)*{}**\dir{-};
	(0,0)*{};(0,2)*{}**\dir{-};(1.5,1.5)*{\scriptstyle 3};
	\endxy}))
=\raisebox{-.1mm}{\xy 0;/r.3pc/:
	(-1.5,0)*{};(-1.5,2)*{}**\dir{-};
	(-1.5,0)*{\sbu};(-1.5,2)*{\sbu};(1,0)*{\sbu};
	(-3,0)*{\scriptstyle 1};
	(-3,2)*{\scriptstyle 2};
	(2.5,0)*{\scriptstyle 3};
	\endxy}
-\raisebox{-.1mm}{\xy 0;/r.3pc/:
	(1.5,0)*{};(0,2)*{}**\dir{-};
	(1.5,0)*{};(3,2)*{}**\dir{-};
	(1.5,0)*{\sbu};(0,2)*{\sbu};(3,2)*{\sbu};
	(0,0)*{\scriptstyle 1};
	(0,3.5)*{\scriptstyle 2};
	(3,3.5)*{\scriptstyle 3};
	\endxy}.\]

Since $\caly(V)$ is also freely generated by $\calw(E)=\Big\{\,|\vee_v Y\,|\,Y\in\calb(E),v\in E\Big\}$ as an algebra by Theorem~\mref{fabt}, we can similarly define an algebra homomorphism
\[\Theta_\sigma:\caly(V)\rar\calr(V),\,\Theta_\sigma(\,|\,)=1,\,\Theta_\sigma(\,|\vee_v Y)=(B^+_v\circ\Theta_\sigma)(Y),\,Y\in\bar{\calb}(V),v\in V.\]

\begin{theorem}
Given any braided vector space $(V,\sigma)$, the map $\Theta_\sigma$ is a graded braided Hopf algebra isomorphism between $\caly(V)$ and $\calr(V)$.
\mlabel{thm:iso}
\end{theorem}
\begin{proof}
First like $\Theta$, $\Theta_\sigma$ is bijective with its inverse defined recursively by
$\Theta_\sigma^{-1}(1)=|\,,\,\Theta_\sigma^{-1}(T)=|\vee_v\Theta_\sigma^{-1}(F)$,
for any $T=B^+_v(F)\in\bar{\calt}(V)$ for some $F\in\bar{\calf}(V),v\in V$.  Also, it is clear that $\Theta_\sigma$ is homogeneous with respect to the degrees $d_v$ on $\caly(V)$ and $\calr(V)$.

We next check that $\sigma_{RT}(\Theta_\sigma\otimes\Theta_\sigma)=(\Theta_\sigma\otimes\Theta_\sigma)\sigma_{BT}$. Thus $\Theta_\sigma:\caly(V)\rar\calr(V)$ is a braided algebra map. It can be done for the generators in $\calw(E)$ by induction on $d_v$. For any $Y,Y'\in\bar{\calb}(V),\,u,v\in V$,
\begin{align*}
\sigma_{RT}&(\Theta_\sigma(\,|\vee_u Y)\otimes\Theta_\sigma(\,|\vee_v Y'))
=\sigma_{RT}(B^+_u\otimes B^+_v)(\Theta_\sigma(Y)\otimes\Theta_\sigma(Y'))\\
&=(B^+\otimes B^+)(\sigma_{V,\calr(V)})_2\sigma_1(\sigma_{RT})_3(\sigma_{\calr(V),V})_2(u\otimes \Theta_\sigma(Y)\otimes v\otimes \Theta_\sigma(Y'))\\
&=(B^+\otimes B^+)(\Id_V\otimes \Theta_\sigma\otimes\Id_V\otimes \Theta_\sigma)(\sigma_{V,\caly(V)})_2\sigma_1(\sigma_{BT})_3(\sigma_{\caly(V),V})_2(u\otimes Y\otimes v\otimes Y')\\
&=(\Theta_\sigma\otimes\Theta_\sigma)\sigma_{BT}((\,|\vee_u Y)\otimes(\,|\vee_v Y')),
\end{align*}
where the second equality applies Eqs.~\meqref{fv1} and \meqref{fv2}; the third one is due to the induction hypothesis and the fact that $\Theta$ preserves the canonical orders for vertices of trees shown in Lemma \mref{ord}; and the last one is by Eqs.~\meqref{tv1} and \meqref{tv2}.

To complete the proof, we only need to show that $\Theta_\sigma$ is a braided coalgebra homomorphism for which we apply induction on $d_v$. It is clear that $\Delta'_{RT}\Theta_\sigma(\,|\,)=1\otimes 1=(\Theta_\sigma\otimes\Theta_\sigma)\Delta'_{BT}(\,|\,)$ by definition.  Further, for any $Y\in\bar{\calb}(V)$ and $v\in V$, we have
\begin{align*}
\Delta'_{RT}\Theta_\sigma&(\,|\vee_v Y)
=\Delta'_{RT}B^+_v\Theta_\sigma(Y)\\
&=B^+_v(\Theta_\sigma(Y))\otimes 1+(\Id_{\calr(V)}\otimes B^+)(\sigma_{V,\calr(V)})_1(v\otimes \Delta'_{RT}\Theta_\sigma(Y))\\
&=\Theta_\sigma(\,|\vee_vY)\otimes 1+(\Id_{\calr(V)}\otimes B^+)(\sigma_{V,\calr(V)})_1(v\otimes (\Theta_\sigma\otimes\Theta_\sigma)\Delta'_{BT}(Y))\\
&=\Theta_\sigma(\,|\vee_v Y)\otimes 1+(\Theta_\sigma\otimes B^+(\Id_V\otimes\Theta_\sigma))(\sigma_{V,\caly(V)})_1(v\otimes\Delta'_{BT}(Y))\\
&=\Theta_\sigma(\,|\vee_v Y)\otimes \Theta_\sigma(\,|\,)+(\Theta_\sigma\otimes\Theta_\sigma\vee)(\sigma_{BT})_1(\sigma_{V,\caly(V)})_2(\,|\otimes v\otimes\Delta'_{BT}(Y))\\
&=(\Theta_\sigma\otimes\Theta_\sigma)\Delta'_{BT}(\,|\vee_v Y),
\end{align*}
where the first and fifth equalities are by the definition of $\Theta_\sigma$; the second equality uses the twisted 1-cocycle condition \meqref{bcoc} of $B^+$;
the third equality is due to the induction hypothesis
and the fourth one is based on the fact that $\Theta$ preserves the canonical orders for vertices of trees.
\end{proof}

\section{Braided tridendriform algebra of planar angularly decorated trees}
\mlabel{sec:trid}
We now extending our braiding approach to tridendriform algebra and the second Hopf algebra of Loday and Ronco on angularly decorated planar rooted trees as free tridendriform algebras. The free objects can be generated by a (braided) space or a (braided) algebra. In the later case, the combinatorial carrier is a special class of planar rooted trees of independent interest.

\subsection{Braided tridendriform algebras}
We recall another important notion of Loday and Ronco~\mcite{LR1}.
\begin{defn}
A quadruple $(R,\prec,\succ,\cdot)$ is
called a {\rm tridendriform algebra} if $R$ is a $\bk$-module with three binary operations $\prec,\succ,\cdot$ satisfying the relations
\begin{align}
&\mlabel{tprec}(x\prec y)\prec z=x\prec(y* z),\\
&\mlabel{tps}(x\succ y)\prec z= x\succ (y \prec z),\\
&\mlabel{tsucc}x\succ (y\succ z)=(x* y)\succ z,\\
&\mlabel{tpc} (x\cdot y)\prec z= x\cdot (y\prec z),\\
&\mlabel{tpsc} (x\prec y)\cdot z= x\cdot (y\succ z),\\
&\mlabel{tsc} (x\succ y)\cdot z= x\succ(y \cdot z),\\
&\mlabel{tc}(x\cdot y)\cdot z=x\cdot(y\cdot z),
\end{align}
for $x,y,z\in R$, where $*:=\prec+\succ+\,\cdot$ is henceforth associative.
\mlabel{tri}
\end{defn}

Define the {\bf augmentation} $R^+:=\bk\oplus R$ of $R$ with $\prec,\succ,\cdot$ satisfying
\begin{equation}
1\prec x=0,\,x\prec 1=x, \quad 1\succ x=x, x\succ 1=0,
1\cdot x=x\cdot 1=0 \tforall x\in R.
\mlabel{us}
\end{equation}
The operation $*$ is extended to $R^+$ with unit 1, while $1\prec1,\,1\succ1,\,1\cdot 1$ are undefined.
	
Furthermore, a submodule $I$ of $R$ is called a {\bf tridendriform ideal}, if
\[R \prec I\subseteq I,\,I \prec R\subseteq I,\,R \succ I\subseteq I,\,I \succ R\subseteq I,\,R \cdot I\subseteq I,\,I \cdot R\subseteq I.\]
Then the quotient module $R/I$ becomes a tridendriform algebra with induced operators $(\prec,\succ,\cdot)$.
For tridendriform algebras $(R,\prec,\succ,\cdot)$ and $(R',\prec',\succ',\cdot')$, a map $f:R\rightarrow R'$ is called a {\bf homomorphism of tridendriform algebras} if
$f$ is a linear map such that $f\prec=\prec'(f\otimes f)$, $f\succ=\succ'(f\otimes f)$ and $f\cdot=\cdot'(f\otimes f)$.

Next we introduce the braided analogue of tridendriform algebras~\mcite{GL}.
\begin{defn}\mlabel{btri}
A quintuple $(R,\prec,\succ,\cdot,\sigma)$ is called a {\bf braided tridendriform algebra} if $(R,\sigma)$ is a braided vector space and $(R,\prec,\succ,\cdot)$ is a tridendriform algebra such that
\begin{align}
&\mlabel{bta1}\sigma(\Id_R\otimes\prec)=(\prec\otimes\Id_R)\sigma_2\sigma_1,\quad
\sigma(\prec\otimes\Id_R)=(\Id_R\otimes\prec)\sigma_1\sigma_2,\\
&\mlabel{bta2}\sigma(\Id_R\otimes\succ)=(\succ\otimes\Id_R)\sigma_2\sigma_1,\quad
\sigma(\succ\otimes\Id_R)=(\Id_R\otimes\succ)\sigma_1\sigma_2,\\
&\mlabel{bta3}
\sigma(\Id_R\otimes\cdot)=(\cdot\otimes\Id_R)\sigma_2\sigma_1,\,\,\,\quad
\sigma(\cdot\otimes\Id_R)=(\Id_R\otimes\cdot)\sigma_1\sigma_2.
\end{align}
\end{defn}
Let $*:=\prec+\succ+\,\cdot$\,. Then $(R,*,\sigma)$ is a braided algebra. Also, $(R^+,*,\sigma)$ is a unital braided algebra with the braiding $\sigma$ of $R$ extended by
\[\sigma(1\otimes x)=x\otimes 1,\,\sigma(x\otimes 1)=1\otimes x,\,x\in R^+.\]
	
For two braided tridendriform algebras $(R,\prec,\succ,\sigma)$ and $(R',\prec',\succ',\sigma')$, a map $f:R\rightarrow R'$ is called a {\bf homomorphism of braided tridendriform algebras} if $f$ is a homomorphism of tridendriform algebras and $(f\otimes f)\sigma=\sigma'(f\otimes f)$.	
Moreover, a tridendriform ideal $I$ of $R$ is called a {\bf braided tridendriform ideal} if it is also a braided subspace of $R$. Then the tridendriform quotient algebra $R/I$ is also a braided tridendriform algebra.

\subsection{Free braided tridendriform algebras on braided vector spaces}

In \mcite{LR2}, Loday and Ronco used planar rooted trees to construct free tridendriform algebras, providing motivation to the notion of angularly decorated planar rooted forests introduced in~\mcite{EG2,Gu2} to construct free Rota-Baxter algebras. The notion of angle of a planar rooted tree can be tracked to the work~\mcite{Ko} of Kontsevich.

Let $\cald$ be the set of planar angularly decorated rooted trees with valency of all internal vertices greater than 2. Let $\cald_n$ be the set of trees in $\cald$ with $n+1$ leaves, $n\geq0$, and define $d_l(T):=n$ to be the {\bf leaf degree} of $T\in \cald_n,\,n\geq0$. In particular, $\cald_0=\{\,|\,\}$.
By Schr\"{o}der's second problem in \cite[Example 6.2.8]{Sta}, we know that $|\cald_n|$
is equal to the $(n+1)$-th little Schr\"{o}der number (or super Catalan number) $s_{n+1}$, $n\geq0$.
On the other hand, by~\cite[\S 4.1]{Gu1} we
can interpret $\cald$ as a subset of $\calt$ by adding leaf vertices and removing the root edge for any tree in $\cald$.

Let $\cald(X)$ be the set consisting of trees in $\cald$ angularly decorated by a set $X$, and we have $\cald(X)=\sqcup_{n\geq0}\cald_n(X)$. See~\cite[\S 4.2.1.3]{Gu1} and \cite[\S 5.1]{Gu2} for such decorations. Some angularly decorated trees with small leaf degrees are
\[|\,,\,\raisebox{-.1mm}{\xy 0;/r.3pc/:
	(1.5,4)*{};(0,2)*{}**\dir{-};
	(-1.5,4)*{};(0,2)*{}**\dir{-};
	(0,0)*{};(0,2)*{}**\dir{-};(0,4)*{\scriptstyle x};
	\endxy}\,,\,\raisebox{-.1mm}{\xy 0;/r.3pc/:
	(3,6)*{};(0,2)*{}**\dir{-};
	(-3,6)*{};(0,2)*{}**\dir{-};
	(-1.5,4)*{};(0,6)*{}**\dir{-};
	(0,0)*{};(0,2)*{}**\dir{-};
	(0,4)*{\scriptstyle y};
	(-1.5,6)*{\scriptstyle x};
	\endxy}\,,\,\raisebox{-.1mm}{\xy 0;/r.3pc/:
	(3,6)*{};(0,2)*{}**\dir{-};
	(-3,6)*{};(0,2)*{}**\dir{-};
	(1.5,4)*{};(0,6)*{}**\dir{-};
	(0,0)*{};(0,2)*{}**\dir{-};
	(0,4)*{\scriptstyle x};
	(1.5,6)*{\scriptstyle y};
	\endxy}\,,\,\raisebox{-.1mm}{\xy 0;/r.3pc/:
	(3,6)*{};(0,2)*{}**\dir{-};
	(-3,6)*{};(0,2)*{}**\dir{-};
	(0,2)*{};(0,6)*{}**\dir{-};
	(0,0)*{};(0,2)*{}**\dir{-};
	(-1.5,6)*{\scriptstyle x};
	(1.5,6)*{\scriptstyle y};
	\endxy}\,,\,\raisebox{-.1mm}{\xy 0;/r.3pc/:
	(4,6)*{};(0,2)*{}**\dir{-};
	(-4,6)*{};(0,2)*{}**\dir{-};
	(2,4)*{};(1,6)*{}**\dir{-};
	(-2,4)*{};(-1,6)*{}**\dir{-};
	(0,0)*{};(0,2)*{}**\dir{-};
	(0,4)*{\scriptstyle y};
	(-2.5,6)*{\scriptstyle x};
	(2.5,6)*{\scriptstyle z};
	\endxy}\,,\,\raisebox{-.1mm}{\xy 0;/r.3pc/:
	(4.5,6)*{};(0,2)*{}**\dir{-};
	(-4.5,6)*{};(0,2)*{}**\dir{-};
	(0,2)*{};(1.5,6)*{}**\dir{-};
	(0,2)*{};(-1.5,6)*{}**\dir{-};
	(0,0)*{};(0,2)*{}**\dir{-};
	(-3,6)*{\scriptstyle x};
	(0,6)*{\scriptstyle y};
	(3,6)*{\scriptstyle z};
	\endxy}\,,\dots\]
where $x,y,z\in X$. We will use the notation
$(T;x_1,\dots,x_n)\in \cald_n(X)$ to denote $T\in \cald_n$ with angular decorations by $x_1,\dots,x_n\in X$ from left to right.

Also, note that any planar rooted tree $T\in\cald(X)$ can be uniquely
obtained by jointing the $k+1$ roots of its branches $T_0,\dots,T_k\in\cald(X)$ to a new vertex with angular decoration $x_1,\dots,x_k\in X$ from left to right and adding a new root. We denote this representation of $T$ by
\begin{equation}\mlabel{eq:rep}
T_0\vee_{x_1}\cdots\vee_{x_k} T_k.
\end{equation}
In particular, we denote the tree $\raisebox{-.1mm}{\xy 0;/r.3pc/:
	(1.5,4)*{};(0,2)*{}**\dir{-};
	(-1.5,4)*{};(0,2)*{}**\dir{-};
	(0,0)*{};(0,2)*{}**\dir{-};(0,4)*{\scriptstyle x};
	\endxy}=|\vee_x|$
by $T[x]$.

As in the cases of planar binary trees $\caly(V)$ and of planar rooted forests $\calr(V)$, we define
\[\calp(V):=\bigoplus_{n\geq0}\calp_n(V), \quad \calp_n(V):=
\bk\cald_n\otimes V^{\otimes n},\]
and $\oline{\calp(V)}:=\bigoplus_{n\geq1}\calp_n(V)$ as its subspace.

\smallskip
Also, identify $T\otimes(v_1\otimes\cdots\otimes v_n)$ with $(T;v_1,\dots,v_n)$, and denote
\[\bar{\cald}_n(V):=\{(T;v_1,\dots,v_n)\,|\,T\in\cald_n,\,v_1,\dots,v_n\in V\},\]
for any $n\geq0$. Let $\bar{\cald}(V):=\sqcup_{n\geq0}\bar{\cald}_n(V)$ with $\bar{\cald}_0(V)=\{\,|\,\}$.
Then $\bar{\cald}(V)$ linearly spans $\calp(V)$.

Given a braided vector space $(V,\sigma)$, the braiding $\sigma$ on $V$ induces a braiding $\sigma_{AT}$ on $\calp(V)$ defined by
\begin{align*}
&\sigma_{AT}((T;v_1,\dots,v_m)\otimes|\,)=|\otimes (T;v_1,\dots,v_m),\,
\sigma_{AT}(\,|\otimes (T;v_1,\dots,v_m))=(T;v_1,\dots,v_m)\otimes|\,,\\
&\sigma_{AT}((T;v_1,\dots,v_m)\otimes (T';v_{m+1},\dots,v_{m+n}))=(T'\otimes T)\beta_{mn}(v_1\otimes\cdots\otimes v_{m+n}),
\end{align*}
for any $T\in\cald_m,T'\in\cald_n$ and $v_1,\dots,v_{m+n}\in V$. Also, we define the linear map
\[\sigma_{V,\calp(V)}:V\otimes\calp(V)\rar \calp(V)\otimes V,\,v\otimes (T;v_1,\dots,v_m)\mapsto (T\otimes\Id_V)\beta_{1,\,m}(v\otimes v_1\otimes\cdots\otimes v_m)\]
for any $T\in\cald_m$ and $v,v_1,\dots,v_m\in V$, and define $\sigma_{\calp(V),V}:\calp(V)\otimes V\rar V\otimes\calp(V)$ by symmetry.

For $k\geq 1$, define
\[\vee=\vee_k:(\calp(V)\otimes V)^{\otimes k}\otimes \calp(V)\rar\calp(V),\,T_0\otimes v_1\otimes\cdots\otimes v_k\otimes T_k\mapsto T_0\vee_{v_1}\cdots\vee_{v_k} T_k\]
for any $T_0,\dots,T_k\in\bar{\cald}(V), v_1,\dots,v_k\in V$, to be the (linear) grafting operation on $\calp(V)$.
The following result is obtained by the definitions of the operator $\vee$ and the braidings $\sigma_{AT},\,\sigma_{V,\calp(V)},\,\sigma_{\calp(V),V}$.

\begin{lemma}
	For any $v_1,\dots,v_{k-1}\in V,\,T,T_1,\dots,T_k\in\bar{\cald}(V)$, we have
	\begin{align}
	&\mlabel{av1}\sigma_{AT}((T_1\vee_{v_1}\cdots\vee_{v_{k-1}} T_k)\otimes T)\\
	&\nonumber \hspace{.2cm}=(\Id_{\calp(V)}\otimes\vee)(\sigma_{AT})_1(\sigma_{V,\calp(V)})_2\cdots(\sigma_{V,\calp(V)})_{2(k-1)}(\sigma_{AT})_{2k-1}(T_1\otimes v_1\otimes\cdots\otimes v_{k-1}\otimes T_k\otimes T),\\
	&\mlabel{av2}
	\sigma_{AT}(T\otimes(T_1\vee_{v_1}\cdots\vee_{v_{k-1}} T_k))\\
	&\nonumber \hspace{.2cm}=(\vee\otimes\Id_{\calp(V)})(\sigma_{AT})_{2k-1}(\sigma_{\calp(V),V})_{2(k-1)}\cdots(\sigma_{\calp(V),V})_2(\sigma_{AT})_1(T\otimes T_1\otimes v_1\otimes\cdots\otimes v_{k-1}\otimes T_k).
	\end{align}	
\end{lemma}	

Define linear maps $\prec_\sigma,\succ_\sigma,\cdot_\sigma:\calp(V)\otimes \calp(V)\rar\calp(V)$ by the following recursions.
\begin{align}
&\mlabel{ap}T\prec_\sigma |=T,\,|\prec_\sigma T=0,\,T\prec_\sigma T'=T_0\vee_{v_1}\cdots\vee_{v_k}(T_k*_\sigma T');\\
&\mlabel{as}T\succ_\sigma |=0,\,|\succ_\sigma T=T,\,T\succ_\sigma T'=(T*_\sigma T_0')\vee_{v_{k+1}}\cdots\vee_{v_{k+l}} T_l';\\
&\mlabel{ac}T\cdot_\sigma |=|\cdot_\sigma T=0,\,T\cdot_\sigma T'=T_0\vee_{v_1}\cdots\vee_{v_k}(T_k*_\sigma T'_0)\vee_{v_{k+1}}\cdots\vee_{v_{k+l}}T'_l,
\end{align}
for any $T=T_0\vee_{v_1}\cdots\vee_{v_k} T_k\in\bar{\cald}_m(V),\,T'=T_0'\vee_{v_{k+1}} \cdots\vee_{v_{k+l}} T_l'\in\bar{\cald}_n(V)$ with $v_1,\dots,v_{k+l}\in V,\,1\leq k\leq m,1\leq l\leq n,\,m,n\geq1$,
where $*_\sigma:=\prec_\sigma+\succ_\sigma+\,\cdot_\sigma$ on $\oline{\calp(V)}$ and is extended to $\calp(V)$ with the unit $|\in\bar{\cald}_0(V)$. Let
\begin{equation}
i_V: V\rar\oline{\calp(V)},\,v\mapsto T[v]
\mlabel{eq:vtaemb}
\end{equation}
be the natural embedding of braided vector spaces.

\begin{theorem}
Given a braided vector space $(V,\sigma)$, 
the quintuple $(\oline{\calp(V)},\prec_\sigma,\succ_\sigma,\cdot_\sigma,\sigma_{AT})$ is a  braided tridendriform algebra.
\mlabel{btat}
\end{theorem}
\begin{proof}
First note that $(\oline{\calp(V)},\prec_\sigma,\succ_\sigma,\cdot_\sigma)$ is exactly
the original tridendriform algebra structure $\left(\bigoplus_{n\geq1}\bk\cald_n\otimes V^{\otimes n},\prec,\succ,\cdot\right)$ in \cite[Corollary 2.9]{LR2} where the notation $\bk[T_n]$ is used instead of $\bk\cald_n$.

Now fix $(T;v_1,\dots,v_m)\in\bar{\cald}_m(V),\,(T';v_{m+1},\dots,v_{m+n})\in\bar{\cald}_n(V)$ and $(T'';v_{m+n+1},\dots,v_{m+n+l})\in\bar{\cald}_l(V)$ with $v_1,\dots,v_{m+n+l}\in V,\,m,n,l\geq0$. We still need to check conditions \meqref{bta1}--\meqref{bta3} together with condition~\meqref{ba1}
for $\calp(V)$, for which we apply induction on the leaf degree $d_l$.

The compatibility condition \meqref{bta1} (resp. \meqref{bta2}) between $\prec$ (resp. $\succ$) and $\sigma_{AT}$ can be checked as checking condition~\meqref{bda1} (resp. \meqref{bda2}) in the proof of Theorem \mref{bdbt}.
It remains to verify Eq.~\meqref{bta3}. The case when one of $T,T',T''$ is $|$ is clear. For the other case, suppose that $m,n,l\geq1$, $T=T_0\vee_{v_{i_1}}\cdots\vee_{v_{i_p}}T_p,\,T'=T'_0\vee_{v_{j_1}}\cdots\vee_{v_{j_q}}T'_q$ with $1\leq i_1<\cdots<i_p\leq m<j_1<\cdots<j_q\leq m+n$, then
\begin{align*}
\sigma&_{AT}((T\cdot_\sigma T')\otimes T'')(v_1\otimes\cdots\otimes v_{m+n+l})
=\sigma_{AT}((T_0\vee_{v_{i_1}}\cdots\vee_{v_{i_p}}(T_p*_\sigma T'_0)\vee_{v_{j_1}}\cdots\vee_{v_{j_q}}T'_q)\otimes T'')\\
&=(\Id_{\calp(V)}\otimes\vee)(\sigma_{AT})_1(\sigma_{V,\calp(V)})_2\cdots(\sigma_{V,\calp(V)})_{2(p+q)}(\sigma_{AT})_{2(p+q)+1}\\
&\quad\,(T_0\otimes v_{i_1}\otimes\cdots \otimes v_{i_p}\otimes(T_p*_\sigma T'_0) \otimes v_{j_1}\otimes\cdots \otimes v_{j_q} \otimes T'_q \otimes T'')\\
&=\left(\Id_{\calp(V)}\otimes\vee\left((\Id_{\calp(V)}\otimes\Id_V)^{\otimes p}\otimes*_\sigma\otimes(\Id_V\otimes\Id_{\calp(V)})^{\otimes q}\right)\right)\\
&\quad\,(\sigma_{AT})_1\cdots(\sigma_{AT})_{2p+1}(\sigma_{AT})_{2(p+1)}\cdots(\sigma_{AT})_{2(p+q+1)}\\
&\quad\,(T_0\otimes v_{i_1}\otimes\cdots\otimes v_{i_p}\otimes T_p\otimes T'_0\otimes v_{j_1}\otimes\cdots \otimes v_{j_q} \otimes T'_q \otimes T'')\\
&=(\Id_{\calp(V)}\otimes\cdot_\sigma)(\sigma_{AT})_1(\sigma_{AT})_2(T\otimes T' \otimes T''),
\end{align*}
where the second equality is due to Eq.~\meqref{av1} (with decorations $v_1,\dots,v_{m+n+l}$ suppressed); the third one holds by the induction hypothesis for condition \meqref{ba1} of $*_\sigma$; and the last one is by Eq.~\meqref{beta}.

It is similar to prove that
\[\sigma_{AT}(T\otimes (T'\cdot_\sigma T''))=(\cdot_\sigma\otimes \Id_{\calp(V)})(\sigma_{AT})_2(\sigma_{AT})_1(T\otimes T'\otimes T'').\]

So far we have shown that conditions \meqref{bta1}--\meqref{bta3} hold for $\calp(V)$, which also imply Eq.~\meqref{ba1} as $*_\sigma=\prec_\sigma+\succ_\sigma+\,\cdot_\sigma$ on $\oline{\calp(V)}$. Furthermore, Eq.~\meqref{ba2} is by the definition of $\sigma_{AT}$. Hence $(\oline{\calp(V)},\prec_\sigma,\succ_\sigma,\cdot_\sigma,\sigma_{AT})$ is a braided tridendriform algebra, and thus $(\calp(V),*_\sigma,\sigma_{AT})$ is a braided algebra.
\end{proof}

Next we show that $\oline{\calp(V)}$ is a free object in the category of braided tridendriform algebras, as a braided version of \cite[Theorem 2.6]{LR2}.

\begin{theorem}
	Given a braided vector space $(V,\sigma)$, the quintuple $(\oline{\calp(V)},\prec_\sigma,\succ_\sigma,\cdot_\sigma,\sigma_{AT})$ is the free braided tridendriform algebra on $(V,\sigma)$, characterized by the following universal property: for any braided tridendriform algebra $(R,\prec_R,\succ_R,\cdot_R,\tau)$ and a map $\varphi:V\rar R$ of braided vector spaces, there exists a unique homomorphism $\bar{\varphi}:\oline{\calp(V)}\rar R$ of braided tridendriform algebras such that $\varphi=\bar{\varphi}\circ i_V$.
\mlabel{fbtv}
\end{theorem}
\begin{proof}
Given any braided tridendriform algebra $(R,\prec_R,\succ_R,\cdot_R,\tau)$ and any map $\varphi:V\rar R$ of braided vector spaces, one can define a linear map $\bar{\varphi}:\calp(V)\rar R^+$ recursively as follows.
\begin{align*}
&\bar{\varphi}(\,|\,)=1,\,\bar{\varphi}(T[v_1])=\varphi(v_1),\\
&\bar{\varphi}(T_0\vee_{v_1}\cdots\vee_{v_k}T_k)=\begin{cases}
\bar{\varphi}(T_0)\succ_R\varphi(v_1)\prec_R\bar{\varphi}(T_1), &\text{if }k=1,\\
\bar{\varphi}(T_0)\succ_R((\varphi(v_1)\prec_R\bar{\varphi}(T_1))\cdot_R\varphi(v_2))\prec_R\bar{\varphi}(T_2),&\text{if }k=2,\\
\bar{\varphi}(T_0)\succ_R(\varphi(v_1)\cdot_R\bar{\varphi}(T_1\vee_{v_2}\cdots\vee_{v_{k-1}}T_{k-1})\cdot_R\varphi(v_k))\prec_R\bar{\varphi}(T_k),&\text{if }k>2,
\end{cases}
\end{align*}
for any $T_0,\dots,T_k\in\bar{\cald}(V),\,v_1,\dots,v_k\in V$ with $k\geq1$.
Adapting the proof of \cite[Theorem 2.6]{LR2}, we find that the restriction of $\bar{\varphi}$ to $\oline{\calp(V)}$ is a homomorphism of tridendriform algebras.
	
Next we show
\[(\bar{\varphi}\otimes\bar{\varphi})\sigma_{AT}(T\otimes T')=\tau(\bar{\varphi}\otimes\bar{\varphi})(T\otimes T')\]
for $(T;v_1,\dots,v_m)\in\bar{\cald}_m(V),\,(T';v_{m+1},\dots,v_{m+n})\in\bar{\cald}_n(V)$ by induction on the leaf degree $d_l$ of $T$. If one of $T=|\,$, it is clear by the definition of $\sigma_{AT}$.
If not, there are the three cases of $k=1$, $k=2$ or $k>2$ for $T=T_0\vee_{v_{i_1}}\cdots\vee_{v_{i_k}}T_k$ with $1\leq i_1<\cdots<i_k\leq m$. We only check the most involved case of $k>2$, namely $T=T_0\vee_{v_{i_1}}T^c\vee_{v_{i_k}}T_k$ with $T^c:=T_1\vee_{v_2}\cdots\vee_{v_{k-1}}T_{k-1}$.
\begin{align*}
(\bar{\varphi}&\otimes\bar{\varphi})\sigma_{AT}(T\otimes T')=(\bar{\varphi}\otimes\bar{\varphi})(T'\otimes T)\beta_{mn}(v_1\otimes \cdots\otimes v_{m+n})\\
&=(\bar{\varphi}\otimes(\bar{\varphi}\succ_R(\varphi\cdot_R\,\bar{\varphi}\,\cdot_R\varphi)\prec_R\bar{\varphi}))(T'\otimes T_0\otimes\Id_V\otimes T^c\otimes\Id_V \otimes T_k)\beta_{mn}(v_1\otimes \cdots\otimes v_{m+n})\\
&=(\bar{\varphi}\otimes(\bar{\varphi}\succ_R(\varphi\cdot_R\,\bar{\varphi}\,\cdot_R\varphi)\prec_R\bar{\varphi}))(\sigma_{AT})_1(\sigma_{V,\calp(V)})_2(\sigma_{AT})_3(\sigma_{V,\calp(V)})_4(\sigma_{AT})_5\\
&\quad\,(T_0\otimes v_{i_1}\otimes T^c\otimes v_{i_k} \otimes T_k\otimes T')\\
&=(\Id_R\otimes\succ_R\cdot_R\,\cdot_R\prec_R)\tau_1\tau_2\tau_3\tau_4\tau_5(\bar{\varphi}\otimes\varphi\otimes\bar{\varphi}\otimes\varphi\otimes\bar{\varphi}\otimes\bar{\varphi})(T_0\otimes v_{i_1}\otimes T^c\otimes v_{i_k} \otimes T_k\otimes T')\\
&=\tau((\bar{\varphi}\succ_R(\varphi\cdot_R\,\bar{\varphi}\,\cdot_R\varphi)\prec_R\bar{\varphi})\otimes\bar{\varphi})(T_0\otimes v_{i_1}\otimes T^c\otimes v_{i_k} \otimes T_k\otimes T')=\tau(\bar{\varphi}\otimes\bar{\varphi})(T\otimes T'),
\end{align*}
where the first and third equalities are by the definitions of $\sigma_{AT}$ and $\sigma_{V,\calp(V)}$; the second and last equalities are by the definition of $\bar{\varphi}$; the fourth equality uses the induction hypothesis and the fifth one is due to Eqs.~\meqref{bta1}--\meqref{bta3} on $R$. Hence, we have $(\bar{\varphi}\otimes\bar{\varphi})\sigma_{AT}=\tau(\bar{\varphi}\otimes\bar{\varphi})$.
	
By definition, $\bar{\varphi}$ is uniquely determined by its images on $T[v]$ for all $v\in V$ as a homomorphism of tridendriform algebras with $\bar{\varphi}(\,|\,)=1$, while $\varphi=\bar{\varphi}\circ i_V$ means that $\varphi(v)=\bar{\varphi}(T[v])$ for any $v\in V$. Hence, $(\oline{\calp(V)},\prec_\sigma,\succ_\sigma,\cdot_\sigma,\sigma_{AT})$ with $i_V:V\rar\oline{\calp(V)}$ has the desired universal property.
\end{proof}

\subsection{Free braided tridendriform algebras on braided algebras}
Throughout this subsection, we fix a braided algebra
$(A,\cdot_A,\sigma)$, and abbreviate the multiplication $a\cdot_A b\in A$ as $ab$ for simplicity.

Given a set $X$, a tree in $\cald(X)$ is called {\bf \lush} if for any vertex of the tree, its leaf branch can only appear as the leftmost or rightmost branch. In other words, the tree does not have two adjacent angles over a same vertex that are separated by a leaf. The subset of \lush trees of $\cald(X)$ is denoted by $\calg(X)$. Also denote $\calg$ for the set of \lush trees with no decorations.\\
\indent For example,
$\raisebox{-.1mm}{\xy 0;/r.3pc/:
	(3,6)*{};(0,2)*{}**\dir{-};
	(-3,6)*{};(0,2)*{}**\dir{-};
	(-1.5,4)*{};(0,6)*{}**\dir{-};
	(0,0)*{};(0,2)*{}**\dir{-};
	(0,4)*{\scriptstyle y};
	(-1.5,6)*{\scriptstyle x};
	\endxy}\,,\,\raisebox{-.1mm}{\xy 0;/r.3pc/:
(3,6)*{};(0,2)*{}**\dir{-};
(-3,6)*{};(0,2)*{}**\dir{-};
(1.5,8)*{};(0,6)*{}**\dir{-};
(-1.5,8)*{};(0,6)*{}**\dir{-};
(0,2)*{};(0,6)*{}**\dir{-};
(0,0)*{};(0,2)*{}**\dir{-};
(-1.5,6)*{\scriptstyle x};
(1.5,6)*{\scriptstyle z};
(0,8)*{\scriptstyle y};
\endxy}$ are in $\calg(X)$, but $\raisebox{-.1mm}{\xy 0;/r.3pc/:
	(3,6)*{};(0,2)*{}**\dir{-};
	(-3,6)*{};(0,2)*{}**\dir{-};
	(0,2)*{};(0,6)*{}**\dir{-};
	(0,0)*{};(0,2)*{}**\dir{-};
	(-1.5,6)*{\scriptstyle x};
	(1.5,6)*{\scriptstyle y};
	\endxy}$ is not in $\calg(X)$.
Define
\[\calp_A(A):=\bigoplus_{n\geq0}\bk\calg_n\otimes A^{\otimes n}\]
as a braided subspace of $\calp(A)$, and denote $\bar{\calg}(A):=\calp_A(A)\cap\bar{\cald}(A)$ such that $\calp_A(A)=\bk\bar{\calg}(A)$. Then $\calp_A(A)$ has the braided tridendtriform algebra structure $(\prec_\sigma,\succ_\sigma,\cdot_\sigma,\sigma_{AT})$ defined by Eqs.~\meqref{ap} -- \meqref{as} and
the following modification of Eq.~\meqref{ac},
\begin{equation}\mlabel{ac'}
T\cdot_\sigma |=|\cdot_\sigma T=0,\,T\cdot_\sigma T'=
\begin{cases}
T_0\vee_{a_1}\cdots\vee_{a_ka_{k+1}}\cdots\vee_{a_{k+l}}T'_l,&\text{if }T_k,T'_0=|\,,\\
T_0\vee_{a_1}\cdots\vee_{a_k}(T_k*_\sigma T'_0)\vee_{a_{k+1}}\cdots\vee_{a_{k+l}}T'_l,&\text{otherwise},
\end{cases}
\end{equation}
for any $T=T_0\vee_{a_1}\cdots\vee_{a_k} T_k\in\bar{\calg}_m(A),\,T'=T_0'\vee_{a_{k+1}} \cdots\vee_{a_{k+l}} T_l'\in\bar{\calg}_n(A)$ with $a_1,\dots,a_{k+l}\in A,\,1\leq k\leq m,1\leq l\leq n,\,m,n\geq1$.

\begin{prop}
The quintuple $(\calp_A(A),\prec_\sigma,\succ_\sigma,\cdot_\sigma,\sigma_{AT})$ defined by Eqs.~\meqref{ap}, \meqref{as} and \meqref{ac'} is a braided tridendtriform algebra. In particular, $\calp_{\bk^+\langle X\rangle}(\bk^+\langle X\rangle)$
is isomorphic to $\calp(\bk X)$ as braided tridendtriform algebras for any set $X$, where $\bk^+\langle X\rangle$ is the free nonunital algebra on $X$.
\end{prop}
\begin{proof}
First by induction on the leaf degree $d_l$, one can easily see that Eqs.~\meqref{ap}, \meqref{as} and \meqref{ac'} give well-defined operators $\prec_\sigma,\succ_\sigma,\cdot_\sigma$ on $\calp_A(A)$. Next it is also straightforward to check by induction on $d_l$ that they satisfy tridendriform conditions in Eqs.~\meqref{tprec}-\meqref{tc}, while Eq.~\meqref{us} follows from definition. As an instance, we verify
\[(T\succ_\sigma T')\cdot_\sigma T''=T\succ_\sigma(T'\cdot_\sigma T'').\]
Indeed, for $T'=T'_0\vee_{v_1}\cdots\vee_{v_l}T''_l$ with $l>0$, then
\[(T\succ_\sigma T')\cdot_\sigma T''
=((T*_\sigma T'_0)\vee_{v_1}\cdots\vee_{v_l}T'_l)\cdot_\sigma T''=
T\succ_\sigma ((T'_0\vee_{v_1}\cdots\vee_{v_l}T'_l)\cdot_\sigma T'')
=T\succ_\sigma(T'\cdot_\sigma T'')\]
by definition. Otherwise, suppose that $T,T''\neq |\,,\,T'=|$\,, by Eq.~\meqref{us} we still have
\[(T\succ_\sigma T')\cdot_\sigma T''=0=T\succ_\sigma(T'\cdot_\sigma T'').\]	
In particular, when $A=\bk^+\langle X\rangle$, we clearly have a braided vector space embedding
\[j_X:\bk X\rar  \calp_{\bk^+\langle X\rangle}(\bk^+\langle X\rangle),x\mapsto T[x].\]
By the universal property of $\calp(\bk X)$ proved in Theorem \mref{fbtv}, there exists a unique homomorphism $\bar{j}_X:\calp(\bk X)\rar\calp_{\bk^+\langle X\rangle}(\bk^+\langle X\rangle)$ of braided tridendtriform algebras such that $j_X=\bar{j}_X\circ i_{\bk X}$. Then $\bar{j}_X$ is an isomorphism of braided tridendtriform algebras, as it maps $\cald(X)$ to $\calg(\sqcup_{n\geq1}X^n)$ bijectively.
\end{proof}

Let $\oline{\calp_A(A)}:=\bigoplus_{n\geq1}\bk\calg_n\otimes A^{\otimes n}$ as a braided subspace of $\calp_A(A)$ with the product $\cdot_\sigma$ and
\begin{equation}\mlabel{ria}
j_A: A\rar\oline{\calp_A(A)},\,a\mapsto T[a]
\end{equation}
be the natural embedding of braided algebras.
\begin{theorem}
Given a braided algebra $(A,\cdot,\sigma)$, the quintuple $(\oline{\calp_A(A)},\prec_\sigma,\succ_\sigma,\cdot_\sigma,\sigma_{AT})$ is the free braided tridendriform algebra on the braided algebra $(A,\cdot,\sigma)$. More precisely, it satisfies the following universal property: for any braided tridendriform algebra $(R,\prec_R,\succ_R,\cdot_R,\tau)$ and a homomorphism $\psi:(A,\cdot,\sigma)\rar (R,\cdot_R,\tau)$ of braided algebras, there exists a unique homomorphism $\hat{\psi}:\oline{\calp_A(A)}\rar R$ of braided tridendriform algebras such that $\psi=\hat{\psi}\circ j_A$.
\mlabel{fbta}
\end{theorem}
\begin{proof}
Similar to the map $\bar{\varphi}:\calp(V)\rar R^+$ given in Theorem \mref{fbtv}, we define a linear map $\hat{\psi}:\calp_A(A)\rar R^+$ recursively as follows.
\begin{align*}
&\hat{\psi}(\,|\,)=1,\,\hat{\psi}(T[a_1])=\psi(a_1),\\
&\hat{\psi}(T_0\vee_{a_1}\cdots\vee_{a_k}T_k)=\begin{cases}
\hat{\psi}(T_0)\succ_R\psi(a_1)\prec_R\hat{\psi}(T_1),&\text{if }k=1,\\
\hat{\psi}(T_0)\succ_R\left((\psi(a_1)\prec_R\hat{\psi}(T_1))\cdot_R\psi(a_2)\right)\prec_R\hat{\psi}(T_2),&\text{if }k=2,\\
\hat{\psi}(T_0)\succ_R\left(\psi(a_1)\cdot_R\hat{\psi}(T_1\vee_{a_2}\cdots\vee_{a_{k-1}}T_{k-1})\cdot_R\psi(a_k)\right)\prec_R\hat{\psi}(T_k),&\text{if }k>2,
\end{cases}
\end{align*}
for any $T_0,\dots,T_k\in\bar{\calg}(A),\,a_1,\dots,a_k\in A$ with $k\geq2$ such that $T_0\vee_{a_1}\cdots\vee_{a_k}T_k\in\bar{\calg}(A)$.

Note that $\hat{\psi}$ is clearly well-defined and is just the restriction of $\bar{\varphi}$ to $\calp_A(A)$ as braided vector space maps. Next we check that $\hat{\psi}$ is a homomorphism of braided tridendriform algebras. In fact, it follows the same steps as in the proof of \cite[Theorem 2.6]{LR2}.

First we show that
\[\hat{\psi}(T\prec_\sigma T')=\hat{\psi}(T)\prec_\sigma\hat{\psi}(T'),\,T,T'\in\bar{\calg}(A),\]
by induction on the degree $d_l(T)$. When $T$ is $|$\,, it is clear by the definition of $\hat{\psi}$ and Eq.~\meqref{us}. For $T=T_0\vee_{a_1}\cdots\vee_{a_k}T_k,\,k\geq1$,
we can use the same argument to check the cases when $k=1,2$ and $k>2$. For instance, if $k>2$, we have
\begin{align*}
\hat{\psi}(T\prec_\sigma T')&=\hat{\psi}(T_0\vee_{a_1}\cdots\vee_{a_k}(T_k*_\sigma T'))\\
&=\hat{\psi}(T_0)\succ_R\left(\psi(a_1)\cdot_R\hat{\psi}(T_1\vee_{a_2}\cdots\vee_{a_{k-1}}T_{k-1})\cdot_R\psi(a_k)\right)\prec_R\hat{\psi}(T_k*_\sigma T')\\
&=\hat{\psi}(T_0)\succ_R\left(\psi(a_1)\cdot_R\hat{\psi}(T_1\vee_{a_2}\cdots\vee_{a_{k-1}}T_{k-1})\cdot_R\psi(a_k)\right)\prec_R(\hat{\psi}(T_k)*_R\hat{\psi}(T'))\\
&=\left(\hat{\psi}(T_0)\succ_R\left(\psi(a_1)\cdot_R\hat{\psi}(T_1\vee_{a_2}\cdots\vee_{a_{k-1}}T_{k-1})\cdot_R\psi(a_k)\right)\prec_R(\hat{\psi}(T_k)\right)\prec_R\hat{\psi}(T')\\
&=\hat{\psi}(T)\prec_R\hat{\psi}(T'),
\end{align*}
where the first equality is due to Eq.~\meqref{ap}; the second and last equalities are by
the definition of $\hat{\psi}$; the third equality uses the induction hypothesis and the fourth one is by Eq.~\meqref{tprec}.  The identity
\[\hat{\psi}(T\succ_\sigma T')=\hat{\psi}(T)\succ_\sigma\hat{\psi}(T'),\,T,T'\in\bar{\calg}(A),\]
is similar to check.

To show that $\hat{\psi}(T\cdot_\sigma T')=\hat{\psi}(T)\cdot_\sigma\hat{\psi}(T')$ for any $T,T'\in\bar{\calg}(A)$, we first check the case when $T=T_0\vee_{a_1}T_1$ and $T'=T'_0\vee_{a_2}T'_1$. If furthermore $T$ or $T'$ is $|$\,, the verification is easy. If not, then
\begin{align*}
\hat{\psi}(T&\cdot_\sigma T')
=\begin{cases}
\hat{\psi}(T_0\vee_{a_1a_2}T'_1),&\text{if }T_1,T'_0=|\\
\hat{\psi}(T_0\vee_{a_1}(T_1*_\sigma T'_0)\vee_{a_2}T'_1),&\text{otherwise}
\end{cases}\\
&=\begin{cases}
\hat{\psi}(T_0)\succ_R\psi(a_1a_2)\prec_R\hat{\psi}(T'_1),&\text{if }T_1,T'_0=|\\
\hat{\psi}(T_0)\succ_R\left(\left(\psi(a_1)\prec_R\hat{\psi}(T_1*_\sigma T'_0)\right)\cdot_R\psi(a_2)\right)\prec_R\hat{\psi}(T'_1),&\text{otherwise}
\end{cases}\\
&=\begin{cases}
\hat{\psi}(T_0)\succ_R\left(\psi(a_1)\cdot_R\psi(a_2)\right)\prec_R\hat{\psi}(T'_1),&\text{if }T_1,T'_0=|\\
\hat{\psi}(T_0)\succ_R\left(\left(\psi(a_1)\prec_R\left(\hat{\psi}(T_1)*_R \hat{\psi}(T'_0)\right)\right)\cdot_R\psi(a_2)\right)\prec_R\hat{\psi}(T'_1),&\text{otherwise}
\end{cases}\\
&=\begin{cases}
\hat{\psi}(T_0)\succ_R\left(\psi(a_1)\cdot_R\psi(a_2)\right)\prec_R\hat{\psi}(T'_1),&\text{if }T_1,T'_0=|\\
\hat{\psi}(T_0)\succ_R\left(\psi(a_1)\cdot_R \left(\hat{\psi}(T'_0)\succ_R\psi(a_2)\right)\right)\prec_R\hat{\psi}(T'_1),&\text{if }T_1=|\,,T'_0\neq|\\
\hat{\psi}(T_0)\succ_R\left(\left(\psi(a_1)\prec_R\hat{\psi}(T_1)\right)\cdot_R\psi(a_2)\right)\prec_R\hat{\psi}(T'_1),&\text{if }T_1\neq|\,,T'_0=|\\
\hat{\psi}(T_0)\succ_R\left(\left(\left(\psi(a_1)\prec_R\hat{\psi}(T_1)\right)\prec_R \hat{\psi}(T'_0)\right)\cdot_R\psi(a_2)\right)\prec_R\hat{\psi}(T'_1),&\text{otherwise}
\end{cases}\\
&=\begin{cases}
\left(\hat{\psi}(T_0)\succ_R\psi(a_1)\right)\cdot_R\left(\psi(a_2)\prec_R\hat{\psi}(T'_1)\right),&\text{if }T_1,T'_0=|\\
\left(\hat{\psi}(T_0)\succ_R\psi(a_1)\right)\cdot_R \left(\hat{\psi}(T'_0)\succ_R\psi(a_2)\prec_R\hat{\psi}(T'_1)\right),&\text{if }T_1=|\,,T'_0\neq|\\
\left(\hat{\psi}(T_0)\succ_R\psi(a_1)\prec_R\hat{\psi}(T_1)\right)\cdot_R \left(\psi(a_2)\prec_R\hat{\psi}(T'_1)\right),&\text{if }T_1\neq|\,,T'_0=|\\
\left(\hat{\psi}(T_0)\succ_R\psi(a_1)\prec_R\hat{\psi}(T_1)\right)\cdot_R \left(\hat{\psi}(T'_0)\succ_R\psi(a_2)\prec_R\hat{\psi}(T'_1)\right),&\text{otherwise}
\end{cases}\\
&=\hat{\psi}(T)\cdot_R\hat{\psi}(T'),
\end{align*}
where the first equality is due to Eq.~\meqref{ac'}; the second and last equalities are by
the definition of $\hat{\psi}$; the third equality uses the induction hypothesis; the fourth equality is by Eqs.~\meqref{tprec} and \meqref{tpsc}; the fifth one is by Eqs.~\meqref{tps},  \meqref{tpc}--\meqref{tsc}.
Then the general case of $T=T_0\vee_{a_1}\cdots\vee_{a_k}T_k,\,T'=T'_0\vee_{a_{k+1}}\cdots\vee_{a_{k+l}}T'_l,\,k,l\geq2$ follows from a similar argument by induction on the degree $d_l$ of trees.

Further, the identity
$(\hat{\psi}\otimes\hat{\psi})\sigma_{AT}(T\otimes T')=\tau(\hat{\psi}\otimes\hat{\psi})(T\otimes T')$
for $T,T'\in\bar{\calg}(A)$ can be shown in the same way as in the proof of Theorem \mref{fbtv}.

By definition, $\hat{\psi}$ is uniquely determined by its images on $T[a]$ for all $a\in A$ as a homomorphism of tridendriform algebras with $\hat{\psi}(\,|\,)=1$, while $\psi=\hat{\psi}\circ j_A$ gives $\psi(a)=\hat{\psi}(T[a])$ for any $a\in A$. Hence, $(\oline{\calp_A(A)},\prec_\sigma,\succ_\sigma,\cdot_\sigma,\sigma_{AT})$ with $j_A:A\rar\oline{\calp_A(A)}$ has the desired universal property.
\end{proof}

Let $I_A$ be the tridendriform ideal of $\oline{\calp(A)}$ generated by
\[T[a]\cdot_\sigma T[b]-T[a b],\,\forall a,b\in A,\]
then clearly $I_A$ is also an (algebraic) ideal of $\calp(A)$.
Obviously, $I_A$ is a braided subspace of $\oline{\calp(A)}$. Thus we also obtain a braided tridendriform quotient algebra $(\oline{\calp(A)}/I_A,\prec_\sigma,\succ_\sigma,\cdot_\sigma,\sigma_{AT})$.
\begin{coro}
$\oline{\calp_A(A)}$ is isomorphic to $\oline{\calp(A)}/I_A$ as braided tridendriform algebras. In particular, $\oline{\calp(A)}=\oline{\calp_A(A)}\oplus I_A$, thus $\calp(A)=\calp_A(A)\oplus I_A$, as braided vector spaces.
\mlabel{quo}
\end{coro}
\begin{proof}
Since $i_A$ in Eq.~\meqref{eq:vtaemb} has it image contained in $\oline{\calp_A(A)}$, the map $j_A:A\rar \oline{\calp_A(A)}$ given in Eq.~\meqref{ria} is this restriction.
Namely, let $\iota_A:\oline{\calp_A(A)}\rar\oline{\calp(A)}$ be the natural inclusion as braided vector spaces, then $i_A=\iota_A\circ j_A$. By Theorem \mref{fbtv}, there exists a unique homomorphism
\[\bar{j}_A:\oline{\calp(A)}\rar \oline{\calp_A(A)}\]
of braided tridendriform algebras such that $\bar{j}_A(\,|\,)=|$ and $j_A=\bar{j}_A\circ i_A$. In particular, $\bar{j}_A(T[a])=T[a]$ for all $a\in A$, and thus $\bar{j}_A(I_A)=0$ by Eq.~\meqref{ac'} on  $\oline{\calp_A(A)}$. It induces a unique homomorphism
\[\tilde{j}_A:\oline{\calp(A)}/I_A\rar\oline{\calp_A(A)}\]
of braided tridendriform algebras such that $\bar{j}_A=\tilde{j}_A\circ\bar{\pi}_A$,
where $\bar{\pi}_A:\oline{\calp(A)}\rar\oline{\calp(A)}/I_A$ is the canonical projection.

On the other hand, let $\pi_A:=\bar{\pi}_A\circ i_A$, then $\pi_A:A\rar\oline{\calp(A)}/I_A$ is an algebra homomorphism such that $\pi_A(a)=T[a]+I_A$ for $a\in A$ by the definition of $I_A$. Using the universal property of $\oline{\calp_A(A)}$ in Theorem \mref{fbta}, we have a unique homomorphism
\[\hat{\pi}_A:\oline{\calp_A(A)}\rar\oline{\calp(A)}/I_A,\,T\mapsto T+I_A,\,T\in\bar{\calg}(A),\]
of braided tridendriform algebras such that $\hat{\pi}_A(\,|\,)=|+I_A$ and $\pi_A=\hat{\pi}_A\circ j_A$. Therefore, we have
\[(\tilde{j}_A\circ\hat{\pi}_A)\circ j_A=\tilde{j}_A\circ\pi_A=\tilde{j}_A\circ(\bar{\pi}_A\circ i_A)=\bar{j}_A\circ i_A=j_A=\Id_{\oline{\calp_A(A)}}\circ j_A.\]
Again by the universal property of $\oline{\calp_A(A)}$, we know that $\tilde{j}_A\circ\hat{\pi}_A=\Id_{\oline{\calp_A(A)}}$.
Then
\[(\hat{\pi}_A\circ\tilde{j}_A\circ\bar{\pi}_A)\circ i_A=(\hat{\pi}_A\circ\tilde{j}_A)\circ\pi_A=(\hat{\pi}_A\circ\tilde{j}_A)\circ(\hat{\pi}_A\circ j_A)=\hat{\pi}_A\circ j_A=\pi_A=\bar{\pi}_A\circ i_A.\]
Now by the universal property of $\oline{\calp(A)}$ in Theorem \mref{fbtv} instead, we have $\bar{\pi}_A=\hat{\pi}_A\circ\tilde{j}_A\circ\bar{\pi}_A$, thus $\hat{\pi}_A\circ\tilde{j}_A=\Id_{\oline{\calp(A)}/I_A}$ for $\bar{\pi}_A$ is surjective. This means that $\hat{\pi}_A:\oline{\calp_A(A)}\rar\oline{\calp(A)}/I_A$ is an isomorphism of braided tridendriform algebras.
Also, $\hat{\pi}_A=\bar{\pi}_A\circ\iota_A$ by definition. Hence,
\[\bar{j}_A\circ\iota_A=(\tilde{j}_A\circ\bar{\pi}_A)\circ\iota_A=\tilde{j}_A\circ\hat{\pi}_A=\Id_{\oline{\calp_A(A)}},\]
which means $\iota_A$ is a section of $\bar{j}_A$. In particular, $\oline{\calp(A)}=\oline{\calp_A(A)}\oplus I_A$ as braided vector spaces. 	
\end{proof}	
\delete{
In summary, we have the following commutative diagram.
\[\xymatrix@=2.5em{\oline{\calp(A)}\ar@{->}[drr]^-{\bar{\pi}_A}\ar@<-.5ex>@{.>}[rdd]^-{\bar{j}_A}&&\\
&A\ar@{->}[ul]_<<<<{i_A}\ar@{->}[r]^-{\pi_A}\ar@{->}[d]_-{j_A}\ar@{->}[dr]^<<<<<<{\psi}&\oline{\calp(A)}/I_A\ar@<.5ex>@{.>}[dl]^<<<{\tilde{j}_A}|\hole\ar@{.>}[d]^-{\tilde{\psi}}\\
&\oline{\calp_A(A)}\ar@<1.7ex>@{->}[uul]^-{\iota_A}\ar@<.5ex>@{.>}[ru]^<<<{\hat{\pi}_A}|\hole\ar@{.>}[r]^-{\hat{\psi}}&R}\]
}

\subsection{The braided Hopf algebra of planar angularly decorated (\lush) trees}

As noted above, free dendriform algebras can be realized on the space of planar binary rooted trees. Equipped with the coproduct in Eq.~\meqref{lrco}, we have the well-known Loday-Ronco Hopf algebra. Extending such Hopf algebra structure to free tridendriform algebras, Loday~\mcite{Lo2} worked in the more general context of $\calp$-algebras for operads $\calp$ that satisfy certain coherent unit action property. See also~\mcite{EG1}.
We adapt Loday's approach to equip the free braided tridendriform algebra of planar rooted trees with a braided Hopf algebra structure, without going into the full generality of operads as in~\mcite{Lo2}.
\begin{prop}
For any braided tridendriform algebra $(R,\prec,\succ,\cdot,\sigma)$,
denote
\[R\boxtimes R:=(R\otimes\bk)\oplus(\bk\otimes R)\oplus(R\otimes R).\]
Then $R\boxtimes R$ has a braided tridendriform algebra structure $(R\boxtimes R,\prec,\succ,\cdot,\beta)$ defined as follows,
\begin{equation}\mlabel{tpta}
(x\otimes y)\boxdot(x'\otimes y')=
\begin{cases}
(*\otimes\odot)\sigma_2(x\otimes y\otimes x'\otimes y'),&y\text{ or }y'\in R,\\
(x\odot x')\otimes yy'&y,y'\in\bk,
\end{cases}
\end{equation}
where $*:=\prec +\succ +\cdot$ is the product on $R^+$ previously defined in Definition \mref{tri}, and $\boxdot$ (resp. $\odot$) represents any of the operations $\prec,\succ,\cdot$ on $R\boxtimes R$ (resp. $R^+$).
Furthermore, let $\beta:=\beta_{2,2}$ be a braiding on $R^+\otimes R^+$. Then the augmentation $(R\boxtimes R)^+\cong R^+\otimes R^+$ has the braided unital algebra structure as defined in Definition \mref{btri}.
\mlabel{tpt}
\end{prop}
\begin{proof}
First it is a simple though tedious computation to check that the quintuple $(R\boxtimes R,\prec,\succ,\cdot,\sigma)$ defined in the statement is a tridendriform algebra. For instance, to check Eq.~\meqref{tpc},
\begin{align*}
((x\otimes y)&\prec(x'\otimes y'))\cdot(x''\otimes y'')\\
&=\begin{cases}
(*\otimes\cdot)\sigma_2\left(*\,\otimes\prec\otimes\,\Id_R^{\otimes 2}\right)\sigma_2(x\otimes y\otimes x'\otimes y'\otimes x''\otimes y''),&y\text{ or }y'\in R,\\
((x\prec x')\otimes yy')\cdot(x''\otimes y''),&y,y'\in\bk,
\end{cases}\\
&=\begin{cases}
(*\,(*\otimes\Id_R)\otimes\cdot\,(\prec\otimes\,\Id_R))\sigma_3\sigma_4\sigma_2(x\otimes y\otimes x'\otimes y'\otimes x''\otimes y''),&y\text{ or }y'\in R,\\
((x\prec x')*x'')\otimes(yy'\cdot y''),&y,y'\in\bk,y''\in R,\\
((x\prec x')\cdot x'')\otimes yy'y'',&y,y',y''\in \bk
\end{cases}\\
&=\begin{cases}
(*\,(\Id_R\otimes*)\otimes\cdot\,(\Id_R\,\otimes\succ))\sigma_3\sigma_4\sigma_2(x\otimes y\otimes x'\otimes y'\otimes x''\otimes y''),&y\text{ or }y'\in R,y''\in R,\\
0,&y\text{ or }y'\in R,y''\in\bk,\\
0,&y,y'\in\bk,y''\in R,\\
(x\cdot(x'\succ x''))\otimes yy'y'',&y,y',y''\in \bk
\end{cases}\\
&=\begin{cases}
(*\otimes\cdot)\sigma_2\left(\Id_R^{\otimes2}\otimes*\,\otimes\succ\right)\sigma_4(x\otimes y\otimes x'\otimes y'\otimes x''\otimes y''),&y'\text{ or }y''\in R,\\
(x\otimes y)\cdot((x'\succ x'')\otimes y'y''),&y',y''\in \bk
\end{cases}\\
&=(x\otimes y)\cdot((x'\otimes y')\succ(x''\otimes y'')).
\end{align*}

Also, one can establish the compatibilities in Eqs.~\meqref{bta1}--\meqref{bta3} between the braiding $\beta$ and operations $\prec,\succ,\cdot$ on $R\boxtimes R$ directly via those between $\sigma$ and $\prec,\succ,\cdot$ on $R$.
For example, using the notation $\boxdot$ (resp. $\odot$) to represent any of the operations $\prec,\succ,\cdot$ on $R\boxtimes R$ (resp. $R^+$), we have
\begin{align*}
\beta&\left(((x\otimes y)\boxdot(x'\otimes y'))\otimes(x''\otimes y'')\right)=\\
&=\begin{cases}
\beta\left((*\otimes\odot)\sigma_2\otimes\Id_R^{\otimes 2}\right)(x\otimes y \otimes x'\otimes y'\otimes x''\otimes y''),&y\text{ or }y'\in R,\\
\beta\left(((x\odot x')\otimes yy')\otimes(x''\otimes y'')\right),&y,y'\in\bk
\end{cases}\\
&=\begin{cases}
\left(\Id_R^{\otimes 2}\otimes*\otimes\odot\right)\sigma_2\cdots\sigma_5\sigma_1\cdots\sigma_4\sigma_2(x\otimes y \otimes x'\otimes y'\otimes x''\otimes y''),&y\text{ or }y'\in R,\\
\left(\Id_R^{\otimes 2}\otimes\odot\otimes\Id_R\right)\sigma_2\sigma_3\sigma_1\sigma_2\left(x\otimes x'\otimes x''\otimes y''\otimes yy'\right),&y,y'\in\bk
\end{cases}\\
&=\begin{cases}
\left(\Id_R^{\otimes 2}\otimes(*\otimes\odot)\sigma_2\right)\beta_1\beta_2\left((x\otimes y)\otimes (x'\otimes y')\otimes(x''\otimes y'')\right),&y\text{ or }y'\in R,\\
\left(\Id_R^{\otimes 2}\otimes\odot\otimes\Id_R\right)(\beta\otimes\Id_R)\left((x\otimes x')\otimes (x''\otimes y'')\otimes yy'\right),&y,y'\in\bk
\end{cases}\\
&=\left(\Id_{R\boxtimes R}\otimes\boxdot\right)\beta_1\beta_2\left((x\otimes y)\otimes (x'\otimes y')\otimes(x''\otimes y'')\right).
\end{align*}
Hence, we obtain the desired braided tridendriform algebra $(R\boxtimes R,\prec,\succ,\cdot,\beta)$.
\end{proof}

Now we apply Proposition \mref{tpt} to the case when $R=\oline{\calp(V)}$. Then $\left(\oline{\calp(V)}\boxtimes\oline{\calp(V)}\right)^+\cong \calp(V)\otimes\calp(V)$
is a braided tridendriform algebra.
By Theorem \mref{fbtv}, there exists a unique homomorphism
\[\Delta'_{AT}:\calp(V)\rar\calp(V)\otimes\calp(V)\]
of braided tridendriform algebras satisfying $\Delta'_{AT}(\,|\,)=|\otimes |$ and
$\Delta'_{AT}(T[v])=T[v]\otimes |+|\otimes T[v]$ for any $v\in V$. In particular, $\Delta'_{AT}$ is a braided algebra homomorphism with respect to $*$. Also, define a linear map $\varepsilon'_{AT}:\calp(V)\rar\bk,\, T\mapsto\delta_{T,\,|},\,T\in\bar{\cald}(V)$. Then we have
\begin{theorem}
Given a braided vector space $(V,\sigma)$, the quintuple $(\calp(V),*_\sigma,\Delta'_{AT},\varepsilon'_{AT},\sigma_{AT})$ is a braided Hopf algebra. The coproduct $\Delta'_{AT}$ on $\calp(V)$ is defined by the following recursive formula:
\begin{equation}
\Delta'_{AT}(\,|\,)=|\otimes |\,,\,\Delta'_{AT}(T)=T\otimes |+(*_\sigma,\vee)(\Delta'_{AT}(T_0)\otimes v_1\otimes\cdots\otimes v_k\otimes\Delta'_{AT}(T_k)),
\mlabel{datrf}
\end{equation}
for any $T=T_0\vee_{v_1}\cdots\vee_{v_k}T_k\in\bar{\cald}(V),\,v_1,\dots,v_k\in V$ with $k\geq1$, where $(*_\sigma,\vee):=(*_\sigma^{(k)}\otimes\vee_k)\sigma_{\calp(V)}^{(k)}$
with $*_\sigma^{(k)}:= *_\sigma(*_\sigma\otimes\Id_{\calp(V)})\cdots(*_\sigma\otimes\Id_{\calp(V)}^{\otimes(k-1)})$ and
\[\sigma_{\calp(V)}^{(k)}:=\prod^{\longleftarrow}_{i=1,\dots,k}(\sigma_{AT})_{i+1}(\sigma_{V,\calp(V)})_{i+2}\cdots(\sigma_{AT})_{3i-1}(\sigma_{V,\calp(V)})_{3i}=(\sigma_{AT})_{k+1}(\sigma_{V,\calp(V)})_{k+2}\cdots(\sigma_{AT})_2(\sigma_{V,\calp(V)})_3.\]
\mlabel{tvbh}
\end{theorem}
\begin{proof}
For the first result, we only need to show that the map $\Delta'_{AT}$ is a coproduct on $\calp(V)$. Then $(\calp(V),*_\sigma,\Delta'_{AT},\varepsilon'_{AT},\sigma_{AT})$ is a connected, graded and braided bialgebra, and thus a braided Hopf algebra by Lemma \mref{cbch}. Note that  \[\oline{\calp(V)}\boxtimes\oline{\calp(V)}\boxtimes\oline{\calp(V)}\cong(\calp(V)\otimes\calp(V)\otimes\oline{\calp(V)})\oplus(\calp(V)\otimes\oline{\calp(V)}\otimes\bk\,|\,)\oplus(\oline{\calp(V)}\otimes\bk\,|\otimes\bk\,|\,)\]
is a braided tridendriform algebra by Proposition \mref{tpt}. At the same time,
\[\left(\oline{\calp(V)}\boxtimes\oline{\calp(V)}\boxtimes\oline{\calp(V)}\right)^+\cong\calp(V)\otimes\calp(V)\otimes\calp(V).\]
Since both maps  \[(\Delta'_{AT}\otimes\Id_{\calp(V)})\Delta'_{AT},\,(\Id_{\calp(V)}\otimes\Delta'_{AT})\Delta'_{AT}:\calp(V)\rar\calp(V)\otimes\calp(V)\otimes\calp(V)\]
are homomorphisms of braided tridendriform algebras mapping $|$ to $|\otimes|\otimes|$\,, and $T[v]$ to $T[v]\otimes |\otimes |+|\otimes T[v]\otimes |+T[v]\otimes |\otimes |$ for all $v\in V$, they are equal to each other, by the universal property of $\calp(V)$ in Theorem \mref{fbtv}. Hence, we have proved the coassociativity of $\Delta'_{AT}$.

On the other hand, it is easy to see that $\varepsilon'_{AT}:\calp(V)\rar\bk$ is a homomorphism of braided 
algebras. Moreover, the three homomorphisms
\[(\varepsilon'_{AT}\otimes\Id_{\calp(V)})\Delta'_{AT},\,(\Id_{\calp(V)}\otimes\varepsilon'_{AT})\Delta'_{AT},\,\Id_{\calp(V)}:\calp(V)\rar\calp(V)\]
map $|$ to $|$\,, and $T[v]$ to $T[v]$ for all $v\in V$. Thus again by the universal property of $\calp(V)$ we conclude that they coincide with each other. Thus $\Delta'_{AT}$ is a coproduct on $\calp(V)$.

Next we show that $\Delta'_{AT}$ satisfies the recursive formula \meqref{datrf} by induction on $d_l$ of trees.
In fact, we first note that, for any $k\geq1$ and $T=T_0\vee_{v_1}\cdots\vee_{v_k}T_k\in\bar{\cald}(V),\,v_1,\dots,v_k\in V$, there is
\begin{equation}\mlabel{tr}
T=\begin{cases}
T_0\succ_\sigma T[v_1]\prec_\sigma T_1,&\text{if }k=1,\\
T_0\succ_\sigma((T[v_1]\prec_\sigma T_1)\cdot_\sigma T[v_2])\prec_\sigma T_2,&\text{if }k=2,\\
T_0\succ_\sigma(T[v_1]\cdot_\sigma (T_1\vee_{v_2}\cdots\vee_{v_{k-1}}T_{k-1})\cdot_\sigma T[v_k])\prec_\sigma T_k,&\text{if }k>2.
\end{cases}
\end{equation}
Denote
\[\sigma_{AT}^{(k)}:=\prod^{\longleftarrow}_{i=1,\dots,k}(\sigma_{AT})_{i+1}\cdots(\sigma_{AT})_{3i}=(\sigma_{AT})_{k+1}\cdots(\sigma_{AT})_{3k}\cdots(\sigma_{AT})_2(\sigma_{AT})_3.\]
For the case when $k=1$, we have
\begin{align*}
\Delta'_{AT}(T)&
=(T_0\succ_\sigma T[v_1]\prec_\sigma T_1)\otimes |+(*_\sigma\otimes \succ_\sigma\prec_\sigma)(\sigma_{AT})_2(\sigma_{AT})_3\left(\Delta'_{AT}(T_0)\otimes T[v_1]\otimes\Delta'_{AT}(T_1)\right)\\
&=T\otimes |+(*_\sigma\otimes \vee)(\sigma_{AT})_2(\sigma_{V,\calp(V)})_3\left(\Delta'_{AT}(T_0)\otimes v_1\otimes\Delta'_{AT}(T_1)\right)\\
&=T\otimes |+(*_\sigma, \vee)\left(\Delta'_{AT}(T_0)\otimes v_1\otimes\Delta'_{AT}(T_1)\right),
\end{align*}
where the first equality is due to the identity $\Delta'_{AT}(T[v_1])=T[v_1]\otimes |+|\otimes T[v_1]$, the fact that $\Delta'_{AT}$ is a homomorphism of braided tridendriform algebras and Eqs.~\meqref{ap},
\meqref{as} and \meqref{tpta}. The second equality uses Eq.~\meqref{tr}.

For the case when $k=2$, we similarly have
\begin{align*}
\Delta'_{AT}(T)&
=(T_0\succ_\sigma((T[v_1]\prec_\sigma T_1)\cdot_\sigma T[v_2])\prec_\sigma T_2)\otimes |\\
&\quad+(*_\sigma^{(2)}\otimes \succ_\sigma(\prec_\sigma)\,\cdot_\sigma\prec_\sigma)\sigma_{AT}^{(2)}\left(\Delta'_{AT}(T_0)\otimes T[v_1]\otimes\Delta'_{AT}(T_1)\otimes T[v_2]\otimes\Delta'_{AT}(T_2)\right)\\
&=T\otimes |+(*_\sigma^{(2)}\otimes \vee_2)\sigma_{\calp(V)}^{(2)}\left(\Delta'_{AT}(T_0)\otimes v_1\otimes\Delta'_{AT}(T_1)\otimes v_2\otimes\Delta'_{AT}(T_2)\right),\\
&=T\otimes |+(*_\sigma, \vee)\left(\Delta'_{AT}(T_0)\otimes v_1\otimes\Delta'_{AT}(T_1)\otimes v_2\otimes\Delta'_{AT}(T_2)\right).
\end{align*}

For the case when $k>2$, we also have
\begin{align*}
\Delta'_{AT}(T)&
=(T_0\succ_\sigma(T[v_1]\cdot_\sigma (T_1\vee_{v_2}\cdots\vee_{v_{k-1}}T_{k-1})\cdot_\sigma T[v_k])\prec_\sigma T_k)\otimes |\\
&\quad +(*_\sigma^{(2)}\otimes \succ_\sigma\cdot_\sigma\,\cdot_\sigma\prec_\sigma)\sigma_{AT}^{(2)}\left(\Delta'_{AT}(T_0)\otimes T[v_1]\otimes\Delta'_{AT}(T_1\vee_{v_2}\cdots\vee_{v_{k-1}}T_{k-1})\otimes T[v_k]\otimes\Delta'_{AT}(T_k)\right)\\
&=T\otimes |+(*_\sigma^{(2)}\otimes \succ_\sigma\cdot_\sigma\,\cdot_\sigma\prec_\sigma)\sigma_{AT}^{(2)}\left(\Id_{\calp(V)}^{\otimes3}\otimes\left(*_\sigma^{(k-2)}\otimes\vee_{k-2}\right)\sigma_{\calp(V)}^{(k-2)} \otimes\Id_{\calp(V)}^{\otimes3}\right)\\
&\quad\left(\Delta'_{AT}(T_0)\otimes T[v_1]\otimes\Delta'_{AT}(T_1)\otimes v_2\otimes\cdots\otimes v_{k-1}\otimes\Delta'_{AT}(T_{k-1})\otimes T[v_k]\otimes\Delta'_{AT}(T_k)\right)\\
&=T\otimes |+\left(*_\sigma^{(k)}\otimes (\succ_\sigma\cdot_\sigma\,\cdot_\sigma\prec_\sigma)\left(\Id_{\calp(V)}^{\otimes2}\otimes\vee_{k-2}\otimes\Id_{\calp(V)}^{\otimes2}\right)\right)\tilde{\sigma}_{\calp(V)}^{(k)}\\
&\quad\left(\Delta'_{AT}(T_0)\otimes T[v_1]\otimes\Delta'_{AT}(T_1)\otimes v_2\otimes\cdots\otimes v_{k-1}\otimes\Delta'_{AT}(T_{k-1})\otimes T[v_k]\otimes\Delta'_{AT}(T_k)\right)\\
&=T\otimes |+(*_\sigma^{(k)}\otimes \vee_k)\sigma_{\calp(V)}^{(k)}\left(\Delta'_{AT}(T_0)\otimes v_1\otimes\cdots\otimes v_k\otimes\Delta'_{AT}(T_k)\right)\\
&=T\otimes |+(*_\sigma, \vee)\left(\Delta'_{AT}(T_0)\otimes v_1\otimes\cdots\otimes v_k\otimes\Delta'_{AT}(T_k)\right),
\end{align*}
with
\begin{align*}
\tilde{\sigma}_{\calp(V)}^{(k)}&:=(\sigma_{AT})_{k+1}(\sigma_{AT})_{k+2}(\sigma_{AT})_{k+3}(\sigma_{V,\calp(V)})_{k+4}\cdots(\sigma_{AT})_{3k-3}(\sigma_{V,\calp(V)})_{3k-2}(\sigma_{AT})_{3k-1}(\sigma_{AT})_{3k}\\
&\times\prod^{\longleftarrow}_{i=1,\dots,k-1}(\sigma_{AT})_{i+1}(\sigma_{AT})_{i+2}(\sigma_{AT})_{i+3}(\sigma_{V,\calp(V)})_{i+4}\cdots(\sigma_{AT})_{3i-1}(\sigma_{V,\calp(V)})_{3i},
\end{align*}
where the first equality is due to a similar argument as the previous cases; the second equality uses the induction hypothesis of Eq.~\meqref{datrf}; and the third equality is by Eqs.~\meqref{ba1}, \meqref{av1} and the associativity of $*_\sigma$, and the forth one uses Eq.~\meqref{tr}.

As a result, the coproduct $\Delta'_{AT}$ on $\calp(V)$ satisfies the recursive formula~\meqref{datrf}.
\end{proof}	

There is an explicit formula of $\Delta_{AT}$ on $\bk\cald(X)$, naturally generalizing Eq.~\meqref{lrco1} of $\Delta_{BT}$ on the Loday-Ronco Hopf algebra $\calh_{BT}(X)$. It is given by
\begin{equation}\mlabel{pco}
\Delta_{AT}(T)=\sum_{P\in\calf_T}P^*\otimes(T/P),\,T\in\cald(X),
\end{equation}
where $\calf_T$ denotes the set of {\bf subforest} of $T\in\cald(X)$ obtained in the same way as  $\calf_Y$ in Eq.~\meqref{lrco1} of binary tree $Y\in\calb(X)$,
$P^*\in\bk\cald(X)$ is the multiplication of trees in $P\in\calf_T$ via $*$ from left to right, and $T/P\in\cald(X)$ is obtained by removing $P$ from $T$ analogous to $Y/P'$ with $P'\in\calf_Y$. All decorations on the vertices of $T$ remain at their corresponding positions in $P$ and $T/P$.

Consequently,
for any $T(v_1,\dots,v_n)\in\bar{\cald}_n(V)$, each $P\in\calf_T$ with $d_P$ being the common leaf degree of all terms in $P^*$, determines a unique $w_P\in\frakS_{d_P,n-d_P}$ such that
\[\Delta_{AT}(T)=\sum_{P\in\calf_T}P^*(v_{w_P(1)},\dots,v_{w_P(d_P)})\otimes(Y/P)(v_{w_P(d_P+1)},\dots,v_{w_P(n)}).\]

For the braided analogue $\calp(V)$ of $\bk\cald(X)$, we thus obtain the following explicit formula of $\Delta'_{AT}$, generalizing Eq. \meqref{blrco} of $\Delta'_{BT}$ on $\caly(V)$,
\begin{equation}
\Delta'_{AT}(T)=\sum_{P\in\calf_T}(P^*\otimes(T/P))T^\sigma_{w_P^{-1}}(v_1\otimes\cdots\otimes v_n),\,\forall\, T(v_1,\dots,v_n)\in\bar{\cald}_n(V),v_1,\dots,v_n\in V,\,n\geq1.
\mlabel{bpco}
\end{equation}

Similarly, applying Proposition \mref{tpt} to $R=\oline{\calp_A(A)}$, then $\left(\oline{\calp_A(A)}\boxtimes\oline{\calp_A(A)}\right)^+\cong \calp_A(A)\otimes\calp_A(A)$
is a braided tridendriform algebra.
Note that there is a braided algebra homomorphism
\[\Phi_A:A\rar \calp_A(A)\otimes\calp_A(A),\,a\mapsto T[a]\otimes |+|\otimes T[a].\]
Indeed, by Eqs.~\meqref{ac'} and \meqref{tpta}, we have
\[\Phi_A(ab)=T[ab]\otimes |+|\otimes T[ab]=
(T[a]\otimes |+|\otimes T[a])\cdot_\sigma(T[b]\otimes |+|\otimes T[b])=\Phi_A(a)\cdot_\sigma\Phi_A(b)\]
for any $a,b\in A$.
Hence, by the universal property of $\calp_A(A)$ in Theorem \mref{fbta}, there exists a unique homomorphism
\[\Delta'_{AT}:\calp_A(A)\rar\calp_A(A)\otimes\calp_A(A)\]
of braided tridendriform algebras satisfying $\Delta'_{AT}(\,|\,)=|\otimes |$ and
$\Delta'_{AT}(T[a])=T[a]\otimes |+|\otimes T[a]$ for all $a\in A$. Again, $\Delta'_{AT}$ is a braided algebra homomorphism with the product $*_\sigma$ for $\calp_A(A),\,\calp_A(A)\otimes\calp_A(A)$, and we can define $\bk$-linear map \[\varepsilon'_{AT}:\calp_A(A)\rar\bk,\, T\mapsto\delta_{T,\,|},\,T\in\bar{\calg}(A).\]
\begin{prop}
	Given a braided algebra $(A,\sigma)$, the quintuple $(\calp_A(A),*_\sigma,\Delta'_{AT},\varepsilon'_{AT},\sigma_{AT})$ is a braided Hopf quotient algebra of $\calp(A)$ as stated in Theorem \mref{tvbh}.
\end{prop}
\begin{proof}
	In order to obtain the desired result, we can directly show that the tridendriform ideal $I_A$ of $\calp(A)$ generated by
	\[T[a]\cdot_\sigma T[b]-T[a b],\,\forall a,b\in A,\]
	is a biideal of $\calp(A)$. In fact, by Eq. \meqref{tpta},
    \begin{align*}
	\Delta'_{AT}&(T[a]\cdot_\sigma T[b]-T[ab])
	=\Delta'_{AT}(T[a])\cdot_\sigma \Delta'_{AT}(T[b])-\Delta'_{AT}(T[ab])\\
	&=(T[a]\otimes |+|\otimes T[a])\cdot_\sigma(T[b]\otimes |+|\otimes T[b])-(T[ab]\otimes |+|\otimes T[ab])\\
	&=(T[a]\cdot_\sigma T[b]\otimes-T[ab]) \otimes|+|\otimes(T[a]\cdot_\sigma T[b]-T[ab])\in I_A\otimes \calp(A)+\calp(A)\otimes I_A.
    \end{align*}
	Then applying Corollary \mref{quo} and Theorem \mref{tvbh}, we find that $\calp_A(A)\cong\calp(A)/I_A$ is a braided bialgebra. Since $\calp_A(A)$ is still clearly connected as $\calp(A)$, it is a braided Hopf quotient algebra of $\calp(A)$ by Lemma~\mref{cbch}.
\end{proof}

\subsection{Enumerative combinatorics of \lush trees}
\mlabel{ss:encom}
First we introduce two integer sequences, one for \lush trees and one for \lush forests.

First let  $\calg^0:=\{\,|\,\}$ and
$\calg\calf^0:=\{\,|\,,\,|\,\,|\,\}\subset S(\calg^0)$. Then recursively define $\calg^n$ and $\calg\calf^n$ by
\[\calg^{n+1}:=B^+(\calg\calf^n\setminus\calg^n),\,\calg\calf^{n+1} :=S(\calg^{n+1})\sqcup\Big\{\,|\,F,\,F\,|\,,\,|\,F\,|\,\Big|\,F\in S(\calg^{n+1})\Big\}\subset S(\{\,|\,\}\sqcup \calg^{n+1})\]
for any $n\geq0$, where $B^+$ is the usual grafting operation, and $S(\calg^n)$ is the free semigroup generated by $\calg^n$.
It is clear that the set of \lush trees $\calg=\sqcup_{n\geq0}\calg^n$.
Let $\calg\calf:=\sqcup_{n\geq0}\calg\calf^n$ be the set of {\bf \lush} forests. For any $T\in\calg^n$ (resp. $F\in\calg\calf^n$), we define its {\bf depth} $\de(T)=n$ (resp. $\de(F)=n$).
	
Note that $\calg^n$ is infinite for $n\geq2$. Next we introduce another gradation on $\calg$.
Let $\calg_n:=\calg\cap\cald_n$ be the set of \lush trees in $\cald$ with $n+1$ leaves for any $n\geq0$, then $\calg=\sqcup_{n\geq0}\calg_n$. It is obvious that $\calb_n\subset\calg_n\subset\cald_n$.
Let $g_n:=|\calg_n|$. As any tree has the representation~\meqref{eq:rep}, we have the following recursive formula:
\[g_0=g_1=1,\,g_n=\sum_{(k_1,\dots,k_r)\vDash n+1,\,r>1\atop k_i=1\,\Rightarrow\, i\in\{1,r\} }g_{k_1-1}\cdots g_{k_r-1},\,n\geq2,\]
where $(k_1,\dots,k_r)\vDash n+1$ means that $k_1,\dots,k_r\in\ZZ^{>0}$ and $k_1+\cdots+k_r=n+1$.
Hence, the sequence of $g_n$'s is listed as  $1,1,2,6,20,72,272,1064,4272,\dots$. This integer sequence is not yet included in OEIS. See Remark~\mref{rk:seq}.
\begin{theorem}
The generating function $G(x):=\sum_{k\geq1}g_{k-1}x^k$ of the sequence $\{g_n\}_{n\geq0}$ enumerating \lush trees satisfies the identity
\[2G(x)^2-(2x+1)G(x)+x+x^2=0.\]
Further, we have
\[G(x)=\dfrac{2x+1-\sqrt{1-4x-4x^2}}{4},\]
and
\[g_0=g_1=1,\,g_n=\sum_{r=0}^n\dfrac{2^{r-2}(2r-3)!!}{r!}{r\choose n-r},\,n\geq2,\]
where we use the convention that ${i\choose j}=0$ if $0\leq i<j$.
\end{theorem}
\begin{proof}
Note that each $T\in\calg$ satisfies exactly one of the following four conditions:
\begin{enumerate}[(a)]
\item $T=|\,$;

\item $T=T_0\vee\cdots\vee T_r,\,r\geq1$
with exactly one of $T_0$ and $T_r=|\,$;

\item $T=T_0\vee\cdots\vee T_r,\,r\geq1$
with $T_0,T_r=|\,$;

\item $T=T_0\vee\cdots\vee T_r,\,r\geq1$ such that $T_0,T_r\neq|\,$.
\end{enumerate}

In case (a) it gives the first term $G_1(x):=g_0x=x$ in $G(x)$. For any $T\in\calg_n$ in case (b) with $n\geq2$, exactly one of $T_0,T_r$ is $|\,$, while other branches $T_i,\,1\leq i\leq r-1$, are in $\calg\setminus\{\,|\,\}$. Hence, the number of $T\in\calg_n$ from case (b) is
\[2\sum_{k_1,\dots,k_r\geq2,\,r\geq1\atop k_1+\cdots+k_r=n} g_{k_1-1}\cdots g_{k_r-1},\]
and thus case (b) contributes to $G(x)$ as the power series 
\[G_2(x):=\sum_{n\geq2}
\left(2\sum_{k_1,\dots,k_r\geq2,\,r\geq1\atop k_1+\cdots+k_r=n} g_{k_1-1}\cdots g_{k_r-1}\right)x^{n+1}=2x\sum_{r\geq1}\left(\sum_{k\geq2}g_{k-1}x^k\right)^r=\dfrac{2x(G(x)-x)}{1-(G(x)-x)},\]
for $G(x)-x=\sum_{k\geq2}g_{k-1}x^k$. Similarly, for any $T\in\calg_n$ in case (c) with $n\geq1$, we have $T_0,T_r=|\,$, and other $T_i,\,1\leq i\leq r-1$, are in $\calg\setminus\{\,|\,\}$ if $n\geq3$. Thus case (c) contributes to $G(x)$ as
\[G_3(x):=x^2+\sum_{n\geq3}
\left(\sum_{k_1,\dots,k_{r-1}\geq2,\,r\geq2\atop k_1+\cdots+k_{r-1}=n-1} g_{k_1-1}\cdots g_{k_{r-1}-1}\right)x^{n+1}=x^2\sum_{r\geq1}\left(\sum_{k\geq2}g_{k-1}x^k\right)^{r-1}=\dfrac{x^2}{1-(G(x)-x)}.\]
Meanwhile, for any $T\in\calg_n$ in case (d) with $n\geq3$, since all $T_i$ are in $\calg\setminus\{\,|\,\},\,0\leq i\leq r$ with $r\geq1$, it contributes to $G(x)$ as
\[G_4(x):=\sum_{n\geq3}
\left(\sum_{k_0,\dots,k_r\geq2,\,r\geq1\atop k_0+\cdots+k_r=n+1} g_{k_0-1}\cdots g_{k_r-1}\right)x^{n+1}=\sum_{r\geq1}\left(\sum_{k\geq2}g_{k-1}x^k\right)^{r+1}=\dfrac{(G(x)-x)^2}{1-(G(x)-x)}.\]
	
Combining these four parts, we have
\[G(x)=\sum_{i=1}^4G_i(x)=x+\dfrac{2x(G(x)-x)+x^2+(G(x)-x)^2}{1-(G(x)-x)},\]
which reduces to the identity $2G(x)^2-(2x+1)G(x)+x+x^2=0$,
yielding the solution
\[G(x)=\dfrac{2x+1-\sqrt{1-4x-4x^2}}{4}.\]	
Using the power series expansion \[\sqrt{1-x}=1+\sum_{i\geq1}\dfrac{(1/2)(1/2-1)\cdots(1/2-i+1)}{i!}(-x)^i=
1-x/2-\sum_{r\geq2}\dfrac{(2r-3)!!}{r!}(x/2)^r,\]
we obtain
\[G(x)=\dfrac{2x+1}{4}-\dfrac{1}{4}\left(1-(2x+2x^2)-\sum_{r\geq2}\dfrac{(2r-3)!!}{r!}(2x+2x^2)^r\right)=x+\dfrac{x^2}{2}+\sum_{r\geq2}\dfrac{2^{r-2}(2r-3)!!}{r!}(x(1+x))^r.\]
This is equivalent to saying that
$g_0=g_1=1,\,g_n=\sum_{r=0}^n\dfrac{2^{r-2}(2r-3)!!}{r!}{r\choose n-r},\,n\geq2$.
\end{proof}

\begin{remark}
There is an integer sequence $\{a_m\}_{m\geq1}$, listed as A141200 in~\mcite{OEIS}, with its first 7 terms coinciding with those of $\{g_n\}_{n\geq0}$. It is defined by
\[B(x)=x^2+B(B(x)),\]
where $B(x)=\sum_{m\geq1}a_mx^{2m}$.
Comparing the coefficients of $x^{2m},\,m\geq1$ on both hand side of the above identity of $B(x)$, we have
\[a_m=\sum_{(i_1,\dots,i_{2r})\vDash m}a_ra_{i_1}\cdots a_{i_{2r}},\,m\geq1.\]
The firs few terms of this sequence are $1,1,2,6,20,72,272,1065,4282$.
\mlabel{rk:seq}
\end{remark}

\smallskip

\noindent
{\bf Acknowledgments.}
Y. Li thanks Rutgers University -- Newark for its hospitality during his visit in 2018-2019. This work is supported by Natural Science Foundation of China (Grant Nos. 11501214, 11771142,  11771190) and the China Scholarship Council (No. 201808440068).

\bibliographystyle{amsplain}

\end{document}